\documentclass{amsart}

\usepackage{amsmath}
\usepackage{amsfonts}
\usepackage{amssymb,enumerate}
\usepackage{amsthm}
\usepackage[all]{xy}
\usepackage{hyperref}

\newtheorem{lem}{Lemma}[section]
\newtheorem{cor}[lem]{Corollary}
\newtheorem{prop}[lem]{Proposition}
\newtheorem{thm}[lem]{Theorem}

\newtheorem{Defn}[lem]{Definition}
\newtheorem{Ex}[lem]{Example}
\newtheorem{Question}[lem]{Question}
\newtheorem{Property}[lem]{Property}
\newtheorem{Properties}[lem]{Properties}
\newtheorem{Discussion}[lem]{Remark}
\newtheorem{Construction}[lem]{Construction}
\newtheorem{Notation}[lem]{Notation}
\newtheorem{Fact}[lem]{Fact}
\newtheorem{Notationdefinition}[lem]{Definition/Notation}
\newtheorem{Subprops}{}[lem]
\newtheorem{Para}[lem]{}

\newenvironment{defn}{\begin{Defn}\rm}{\end{Defn}}
\newenvironment{ex}{\begin{Ex}\rm}{\end{Ex}}
\newenvironment{question}{\begin{Question}\rm}{\end{Question}}

\newenvironment{notation}{\begin{Notation}\rm}{\end{Notation}}
\newenvironment{fact}{\begin{Fact}\rm}{\end{Fact}}
\newenvironment{notationdefinition}{\begin{Notationdefinition}\rm}{\end{Notationdefinition}}

\newenvironment{disc}{\begin{Discussion}\rm}{\end{Discussion}}

\newcommand{\cat}[1]{\mathcal{#1}}
\newcommand{\catx}{\cat{X}}
\newcommand{\caty}{\cat{Y}}
\newcommand{\catm}{\cat{M}}
\newcommand{\catv}{\cat{V}}
\newcommand{\catw}{\cat{W}}

\newcommand{\catp}{\cat{P}}

\newcommand{\cati}{\cat{I}}
\newcommand{\cata}{\cat{A}}
\newcommand{\catabel}{\mathcal{A}b}

\newcommand{\catb}{\cat{B}}
\newcommand{\catgi}{\cat{GI}}
\newcommand{\catgp}{\cat{GP}}

\newcommand{\catgic}{\cat{GI}_C}
\newcommand{\catgib}{\cat{GI}_B}

\newcommand{\catgibdc}{\cat{GI}_{\bdc}}

\newcommand{\catgc}{\cat{G}_C}
\newcommand{\catgpc}{\cat{GP}_C}
\newcommand{\catgpb}{\cat{GP}_B}

\newcommand{\catac}{\cat{A}_C}
\newcommand{\catab}{\cat{A}_B}
\newcommand{\catbc}{\cat{B}_C}
\newcommand{\catabdc}{\cat{A}_{\bdc}}
\newcommand{\catbbdc}{\cat{B}_{\bdc}}
\newcommand{\catbb}{\cat{B}_B}

\newcommand{\catic}{\cat{I}_C}
\newcommand{\catibdc}{\cat{I}_{\bdc}}
\newcommand{\catpb}{\cat{P}_B}
\newcommand{\catpc}{\cat{P}_C}

\newcommand{\finrescat}[1]{\operatorname{res}\comp{\cat{#1}}}
\newcommand{\proprescat}[1]{\operatorname{res}\wti{\cat{#1}}}
\newcommand{\finrescatx}{\finrescat{X}}
\newcommand{\finrescaty}{\finrescat{Y}}

\newcommand{\finrescatw}{\finrescat{W}}

\newcommand{\finrescatpcr}{\operatorname{res}\comp{\catpc(R)}}

\newcommand{\finrescatpbr}{\operatorname{res}\comp{\catpb(R)}}

\newcommand{\finrescatgpbr}{\operatorname{res}\comp{\catgpb(R)}}
\newcommand{\propcorescatic}{\operatorname{cores}\wti{\catic}}

\newcommand{\fincorescaticr}{\operatorname{cores}\comp{\catic(R)}}
\newcommand{\fincorescatir}{\operatorname{cores}\comp{\cati(R)}}

\newcommand{\fincorescatibdcr}{\operatorname{cores}\comp{\catibdc(R)}}

\newcommand{\fincorescatgibdcr}{\operatorname{cores}\comp{\catgibdc(R)}}

\newcommand{\finrescatgpr}{\operatorname{res}\comp{\cat{GP}(R)}}
\newcommand{\finrescatpr}{\operatorname{res}\comp{\cat{P}(R)}}
\newcommand{\finrescatgpc}{\operatorname{res}\comp{\catgp_C(R)}}
\newcommand{\proprescatpc}{\operatorname{res}\wti{\catp_C(R)}}

\newcommand{\proprescatp}{\proprescat{P}}

\newcommand{\proprescatx}{\proprescat{X}}

\newcommand{\proprescatw}{\proprescat{W}}
\newcommand{\fincorescat}[1]{\operatorname{cores}\comp{\cat{#1}}}
\newcommand{\propcorescat}[1]{\operatorname{cores}\wti{\cat{#1}}}
\newcommand{\fincorescatx}{\fincorescat{X}}

\newcommand{\fincorescatgir}{\fincorescat{GI(R)}}

\newcommand{\propcorescati}{\propcorescat{I}}

\newcommand{\fincorescaty}{\fincorescat{Y}}
\newcommand{\fincorescatv}{\fincorescat{V}}

\newcommand{\propcorescaty}{\propcorescat{Y}}
\newcommand{\propcorescatv}{\propcorescat{V}}

\newcommand{\catao}{\cata^o}
\newcommand{\catxo}{\catx^o}
\newcommand{\catyo}{\caty^o}

\newcommand{\pd}{\operatorname{pd}}	
	
\newcommand{\gkdim}[1]{\mathrm{G}_{#1}\text{-}\dim}	
	
\newcommand{\id}{\operatorname{id}}

\newcommand{\catpd}[1]{\cat{#1}\text{-}\pd}
\newcommand{\xpd}{\catpd{X}}
\newcommand{\xopd}{\catxo\text{-}\pd}

\newcommand{\wpd}{\catpd{W}}

\newcommand{\catid}[1]{\cat{#1}\text{-}\id}
\newcommand{\yid}{\catid{Y}}
\newcommand{\vid}{\catid{V}}

\newcommand{\pbpd}{\catpb\text{-}\pd}

\newcommand{\ibdcid}{\catibdc\text{-}\id}

\newcommand{\depth}{\operatorname{depth}}

\newcommand{\ann}{\operatorname{Ann}}

\newcommand{\grade}{\operatorname{grade}}

\newcommand{\ext}{\operatorname{Ext}}

\newcommand{\HH}{\operatorname{H}}
\newcommand{\Hom}{\operatorname{Hom}}	
\newcommand{\coker}{\operatorname{Coker}}

\newcommand{\tor}{\operatorname{Tor}}
\newcommand{\im}{\operatorname{Im}}
\newcommand{\shift}{\mathsf{\Sigma}}

\newcommand{\cone}{\operatorname{Cone}}
\newcommand{\Ker}{\operatorname{Ker}}

\newcommand{\aext}{\ext_{\cata}}
\newcommand{\ahom}{\Hom_{\cata}}
\newcommand{\aoext}{\ext_{\catao}}
\newcommand{\aohom}{\Hom_{\catao}}
\newcommand{\xaext}{\ext_{\catx\!\cata}}

\newcommand{\ayext}{\ext_{\cata\caty}}
\newcommand{\avext}{\ext_{\cata\catv}}

\newcommand{\waext}{\ext_{\catw \cata}}

\newcommand{\pcext}{\ext_{\catpc}}
\newcommand{\pbext}{\ext_{\catpb}}
\newcommand{\gpcext}{\ext_{\catgpc}}
\newcommand{\icext}{\ext_{\catic}}
\newcommand{\gpbext}{\ext_{\catgpb}}
\newcommand{\gibdcext}{\ext_{\catgibdc}}
\newcommand{\ibdcext}{\ext_{\catibdc}}
\newcommand{\gicext}{\ext_{\catgic}}
\newcommand{\gpext}{\ext_{\catgp}}
\newcommand{\giext}{\ext_{\catgi}}

\newcommand{\wt}{\widetilde}

\newcommand{\comp}[1]{\widehat{#1}}
\newcommand{\ol}{\overline}
\newcommand{\wti}{\widetilde}

\newcommand{\zz}{\mathbb{Z}}

\newcommand{\from}{\leftarrow}
\newcommand{\xra}{\xrightarrow}
\newcommand{\xla}{\xleftarrow}

\newcommand{\xwacomp}{\vartheta_{\catx \catw \cata}}
\newcommand{\ayvcomp}{\vartheta_{\cata \caty \catv}}

\newcommand{\xaacomp}{\varkappa_{\catx \cata}}
\newcommand{\aaycomp}{\varkappa_{\cata\caty}}

\newcommand{\pccomp}{\varkappa_{\catpc}}

\newcommand{\iccomp}{\varkappa_{\catic}}
\newcommand{\ibdccomp}{\varkappa_{\catibdc}}

\newcommand{\bdc}{B^{\dagger_C}}

\renewcommand{\geq}{\geqslant}
\renewcommand{\leq}{\leqslant}
\renewcommand{\ker}{\Ker}
\renewcommand{\hom}{\Hom}

\begin{document}

\bibliographystyle{amsplain}

\author{Sean Sather-Wagstaff}

\address{Sean Sather-Wagstaff, Department of Mathematical Sciences, Kent State University,
  Mathematics and Computer Science Building, Summit Street, Kent OH
  44242, USA}
\email{sather@math.kent.edu}
\urladdr{http://www.math.kent.edu/~sather}

\author{Tirdad Sharif}
\address{Tirdad Sharif, School of Mathematics, Institute for Studies in
Theoretical Physics and Mathematics, P. O. Box 19395-5746, Tehran, Iran}
\email{sharif@ipm.ir}
\urladdr{http://www.ipm.ac.ir/IPM/people/personalinfo.jsp?PeopleCode=IP0400060}
\thanks{TS is supported by a grant from IPM, (No. 83130311).}

\author{Diana White} 
\address{Diana White, Department of Mathematics, University of Nebraska,
   203 Avery Hall, Lincoln, NE, 68588-0130 USA} \email{dwhite@math.unl.edu}
\urladdr{http://www.math.unl.edu/~s-dwhite14/}

\title{Gorenstein cohomology in abelian categories}

\date{\today}


\keywords{abelian category, Auslander class, balance, 
Bass class, Gorenstein homological
dimensions, injective generator, projective cogenerator,
relative cohomology, relative homological algebra,
semi-dualizing, semidualizing}
\subjclass[2000]{Primary 18G10, 18G15, 18G20, 18G25; Secondary 13D02, 13D05, 13D07}

\begin{abstract}
We investigate relative cohomology functors 
on subcategories of abelian categories
via Auslander-Buchweitz approximations
and the resulting strict resolutions.  
We verify that certain comparison maps between these functors are isomorphisms
and introduce a notion of perfection for this context.
Our main theorem is a balance result 
for relative cohomology that simultaneously
recovers theorems of Holm and the current authors as special cases.
\end{abstract}

\maketitle

\section*{Introduction} 

Let $\cata$ be an abelian category equipped with subcategories
$\catw$ and $\catx$ such that $\catx$ is closed under extensions
and $\catw$ is an injective cogenerator for $\catx$.  
(See Section~\ref{sec1} for definitions and Section~\ref{sec8}
for motivating examples from commutative algebra.)
Given an object
$M$ in $\cata$ with finite $\catx$-projective dimension, 
Auslander and Buchweitz's theory of approximations~\cite{auslander:htmcma}
provides a ``strict $\catw\catx$-resolution''  of $M$.  Such a resolution 
enjoys good enough lifting properties to make it unique up to homotopy
equivalence and, as such, yields a well-defined relative cohomology
functor $\xaext^n(M,-)$ for each integer $n$.  
The functors $\ayext^n(-,N)$ are defined dually.

These functors have been investigated by numerous authors,
beginning with the fundamental work of 
Butler and Horrocks~\cite{butler:cer}
and 
Eilenberg and Moore~\cite{eilenberg:frha}.
Our approach to the subject is based on a fusion
of the techniques of 
Avramov and Martsinkovsky~\cite{avramov:aratc},
Enochs and Jenda~\cite{enochs:rha},
and Holm~\cite{holm:gdf}.

The contents of this paper are summarized as follows.
In Section~\ref{sec9} we present a brief study of
the pertinent  properties of strict resolutions.
Sections~\ref{sec5} focuses on 
conditions guaranteeing that natural comparison
maps are isomorphisms.
In Section~\ref{sec3}  we introduce a notion of relative perfection
and establish a duality between certain classes of relatively perfect objects.  

The main theorem of this paper is the following balance result,
contained in Theorem~\ref{balance02}.  It showcases the benefit
of our  approach to studying these functors, as it 
simultaneously encompasses a result of Holm~\cite[(3.6)]{holm:gdf}
and our own result~\cite[(5.7)]{sather:crct};
see Corollary~\ref{holm1} and Remark~\ref{more07}.

\

\noindent
\textbf{Main Theorem.}
\textit{Let $\catx$, $\caty$, $\catw$ and $\catv$ be subcategories of $\cata$.
Assume that $\catx$ and $\caty$ are  closed under extensions,
$\catw$ is an injective cogenerator for $\catx$, 
$\catv$ is a projective generator for $\caty$,
$\catw\perp\caty$ and $\catx\perp\catv$.
Assume further
$\waext^{\geq1}(T,\catv)=0=\avext^{\geq1}(\catw, U)$
for all objects $T$ and $U$ with
$\wpd(T)<\infty$ and $\vid(U)<\infty$.
If $M$ and $N$ are objects of $\cata$ such that
$\xpd(M)<\infty$ and $\yid(N)<\infty$, then
there are isomorphisms
$\xaext^n(M,N)\cong\ayext^n(M,N)$
for all $n\in\zz$.
}

\section{Categories and Resolutions}\label{sec1}

We begin with some notation and terminology for use throughout this paper.

\begin{notationdefinition} \label{notation01a}
Throughout this work
$\cata$ is an abelian category.
We use the term ``subcategory'' to mean a ``full, additive, and essential 
(closed under isomorphisms) 
subcategory.''
Write $\catp=\catp(\cata)$ and $\cati=\cati(\cata)$ 
for the subcategories of projective and injective
objects in $\cata$, respectively.

We fix subcategories $\catx$, $\caty$, $\catw$, and $\catv$  of $\cata$ such that
$\catw$ is a subcategory of $\catx$
and $\catv$ is a subcategory of $\caty$.
For an object $M\in\cata$, write $M\perp\caty$ (resp., $\catx\perp M$)
if $\aext^{\geq1}(M,Y)=0$ for each object  $Y\in\caty$
(resp., if $\aext^{\geq1}(X,M)=0$ for each object  $X\in\catx$).
Write $\catx\perp\caty$ 
if $\aext^{\geq1}(X,\caty)=0$ for each object $X\in\catx$.
We say that $\catw$ is a \emph{cogenerator} for $\catx$ if,
for each object $X\in\catx$, there exists an exact sequence 
$$0\to X\to W\to X'\to 0$$
with $W\in\catw$ and $X'\in\catx$.
The subcategory
$\catw$ is an \emph{injective cogenerator} for $\catx$ if 
$\catw$ is a cogenerator for $\catx$ and $\catx\perp\catw$.
The terms  \emph{generator} and \emph{projective generator} 
are defined dually.
\end{notationdefinition}

\begin{defn} \label{notation07}
An \emph{$\cata$-complex} is a sequence of 
homomorphisms in $\cata$
$$M =\cdots\xra{\partial^M_{n+1}}M_n\xra{\partial^M_n}
M_{n-1}\xra{\partial^M_{n-1}}\cdots$$
such that $ \partial^M_{n}\partial^M_{n+1}=0$ for each integer $n$; the
$n$th \emph{homology object} of $M$ is
$\HH_n(M)=\Ker(\partial^M_{n})/\im(\partial^M_{n+1})$.
We frequently identify objects in $\cata$ with complexes concentrated in degree 0.
For each integer $i$,
the $i$th \emph{suspension} (or \emph{shift}) of
a complex $M$, denoted $\shift^i M$, is the complex with
$(\shift^i M)_n=M_{n-i}$ and $\partial_n^{\shift^i M}=(-1)^i\partial_{n-i}^M$.
The notation $\shift X$ is short for $\shift^1 X$.

A complex $M$ is \emph{$\ahom(\catx,-)$-exact} if the complex
$\ahom(X,M)$ is exact for each object $X$ in $\catx$.  
The term \emph{$\ahom(-,\catx)$-exact} is defined dually.  
\end{defn}

\begin{defn} \label{notation07a}
Let $M,N$ be $\cata$-complexes.
The  Hom-complex $\ahom(M,N)$ is the complex of abelian groups defined as
$\ahom(M,N)_n=\prod_p\ahom(M_p,N_{p+n})$
with  $\partial_n^{\ahom(M,N)}$ given by
$\alpha=\{\alpha_p\}\mapsto \{\partial^{N}_{p+n} \alpha_p-(-1)^n \alpha_{n-1}\partial^M_p\}$.
A \emph{morphism}  $M\to N$
is an element of $\ker(\partial_0^{\ahom(M,N)})$,
and a morphism is \emph{null-homotopic} if 
it is in $\im(\partial_1^{\ahom(M,N)})$.
Two morphisms $\alpha,\alpha'\colon M\to N$
are \emph{homotopic} 
if $\alpha-\alpha'$ is null-homotopic.  The morphism $\alpha$ is a
\emph{homotopy equivalence} if there is a morphism
$\beta\colon N\to M$ such that 
$\beta \alpha $ is homotopic to $\id_{M}$ and
$\alpha\beta$ is homotopic to $\id_{N}$.

A morphism  $\alpha\colon M\to N$
induces homomorphisms 
$\HH_n(\alpha)\colon\HH_n(M)\to\HH_n(N)$, and $\alpha$ is a
\emph{quasiisomorphism} if each $\HH_n(\alpha)$ is bijective.
The \emph{mapping cone} of $\alpha$ is the complex
$\cone(\alpha)$ defined as
$\cone(\alpha)_n=N_n\oplus M_{n-1}$
and
$\partial^{\cone(\alpha)}_n 
= \Bigl(\begin{smallmatrix}\partial_{n}^{N} & \alpha_{n-1} \\ 0 & -\partial_{n-1}^{M}
\end{smallmatrix} \Bigr)$.
The morphism $\alpha$ is a quasiisomorphism if and only if $\cone(\alpha)$ is exact.
\end{defn}

\begin{defn} \label{notation03}
A complex $X$ is \emph{bounded} if $X_n=0$ for $|n|\gg 0$. When
$X_{-n}=0=\HH_n(X)$ for all $n>0$, the natural morphism
$X\to\HH_0(X)\cong M$ is a quasiisomorphism.  In this event, 
the morphism $X\to M$ is an
\emph{$\catx$-resolution} of $M$ if each $X_n$ is  in $\catx$, and
the exact sequence
$$X^+ = \cdots\xra{\partial^X_{2}}X_1
\xra{\partial^X_{1}}X_0\to M\to 0$$ is the \emph{augmented
$\catx$-resolution} of $M$ associated to $X$.
We write ``projective resolution'' in lieu of
``$\catp$-resolution''.
The \emph{$\catx$-projective dimension} of $M$ is the quantity
$$\xpd(M)=\inf\{\sup\{n\geq 0\mid X_n\neq 0\}\mid \text{$X$ is an
$\catx$-resolution of $M$}\}.$$ 
The objects of $\catx$-projective dimension 0 are
exactly the objects of $\catx$.
We let $\finrescatx$ denote
the subcategory of objects $M$ with $\xpd(M)<\infty$.
One checks easily that $\finrescatx$ is additive and contains $\catx$.

The terms \emph{$\caty$-coresolution} and \emph{$\caty$-injective dimension}
are defined dually.  The \emph{augmented 
$\caty$-coresolution} associated to a $\caty$-coresolution $Y$ is denoted $^+Y$,
and the $\caty$-injective dimension of $M$ is denoted $\yid(M)$.
The subcategory of $R$-modules $N$  with $\yid(N)<\infty$ is denoted
$\fincorescaty$; it is additive and contains $\caty$.
\end{defn}

\begin{defn} \label{notation05}
An $\catx$-resolution $X$ is \emph{proper} if
the augmented resolution $X^+$ is $\ahom(\catx,-)$-exact. 
The subcategory of objects  admitting a proper 
$\catx$-resolution is denoted $\proprescatx$.
One checks readily that $\proprescatx$ is additive and contains
$\catx$.
Projective resolutions are $\catp$-proper, and so $\cata$ has enough 
projectives if and only if $\proprescatp=\cata$.

\emph{Proper coresolutions} are defined dually, and we let
$\propcorescaty$ denote
the subcategory of objects of $\cata$ admitting a proper 
$\caty$-coresolution.
Again, $\propcorescaty$ is additive and contains $\caty$ as a subcategory.
Injective coresolutions are always $\cati$-proper, and so $\cata$ has enough 
injectives if and only if $\propcorescati=\cata$. 
\end{defn}

The next lemmata are standard or have standard proofs:
for~\ref{perp03} see~\cite[pf.~of (2.3)]{auslander:htmcma};
for~\ref{gencat01} see~\cite[pf.~of (2.1)]{auslander:htmcma};
for~\ref{rel01} argue as 
in~\cite[(4.3)]{avramov:aratc} 
or~\cite[pf.~of (8.1.3)]{enochs:rha};
and for the ``Horseshoe Lemma'' \ref{horseshoe01} 
see~\cite[(4.5)]{avramov:aratc}
or~\cite[pf.~of (8.2.1)]{enochs:rha}.

\begin{lem} \label{perp03}
Let $0\to M_1\to M_2\to M_3\to 0$ be an exact sequence in $\cata$.
\begin{enumerate}[\quad\rm(a)]
\item \label{perp03item1}
If $M_3\perp\catx$, then $M_1\perp\catx$ if and only if $M_2\perp\catx$.
If $M_1\perp\catx$  and  $M_2\perp\catx$, 
then $M_3\perp\catx$
if and only if the given sequence is  $\ahom(-,\catx)$ exact.
\item \label{perp03item2}
If $\catx\perp M_1$, then $\catx\perp M_2$ if and only if $\catx\perp M_3$.
If $\catx\perp M_2$  and  $\catx\perp M_3$, 
then $\catx\perp M_1$
if and only if the given sequence is  $\ahom(\catx,-)$ exact. \qed
\end{enumerate}
\end{lem}

\begin{lem} \label{gencat01}
If $\catx\perp\caty$, then  $\catx\perp\finrescaty$ and $\fincorescatx\perp\caty$.
\qed
\end{lem}

\begin{lem} \label{rel01}
Let $M,M',N,N'$ be objects in $\cata$. 
\begin{enumerate}[\quad\rm(a)]
\item \label{rel01item1}
Assume that $M$ admits a proper $\catw$-resolution
$\gamma\colon W\to M$ and $M'$ admits a proper $\catx$-resolution
$\gamma'\colon X'\to M'$.
For each homomorphism $f\colon M\to M'$ there exists
a morphism $\ol{f}\colon W\to X'$ unique up to homotopy
such that $\gamma'\ol{f}=f\gamma$.
If $f$ is an isomorphism, then $\ol{f}$ is a quasiisomorphism.
If $f$ is an isomorphism and $\catx=\catw$, then  $\ol{f}$ is a
homotopy equivalence.
\item \label{rel01item1'}
Assume that $M$ admits a projective resolution
$\gamma\colon P\to M$ and $M'$ admits a proper $\catx$-resolution
$\gamma'\colon X'\to M'$.
For each homomorphism $f\colon M\to M'$ there exists
a morphism $\wti{f}\colon P\to X'$ unique up to homotopy
such that $\gamma'\wti{f}=f\gamma$.
If $f$ is an isomorphism, then $\wti{f}$ is a quasiisomorphism.
\item \label{rel01item2}
Assume that $N$  admits a proper $\caty$-coresolution
$\delta \colon N\to Y$ and $N'$  admits a proper $\catv$-coresolution
$\delta' \colon N'\to V'$.
For each homomorphism $g\colon N\to N'$ there exists
a morphism $\ol{g}\colon Y\to V'$ unique up to homotopy
such that $\ol{g}\delta=\delta' g$.
If $g$ is an isomorphism, then $\ol{g}$ is a quasiisomorphism.
If $g$ is an isomorphism and $\catv=\caty$, then $\ol{g}$  is a
homotopy equivalence. 
\item \label{rel01item2'}
Assume that $N$  admits a proper $\caty$-coresolution
$\delta \colon N\to Y$  and $N'$  admits 
an injective resolution
$\delta' \colon N'\to I'$.
For each homomorphism $g\colon N\to N'$ there exists
a morphism $\wti{g}\colon Y\to I'$ unique up to homotopy
such that $\wti{g}\delta=\delta' g$.
If $g$ is an isomorphism, then $\wti{g}$ is a quasiisomorphism.
\qed
\end{enumerate}
\end{lem}

\begin{lem} \label{horseshoe01}
Let $0\to M'\to M\to M''\to 0$ be an exact sequence in $\cata$.
\begin{enumerate}[\quad\rm(a)]
\item \label{horseshoe01item1}
Assume that $M'$ and $M''$ 
admit proper $\catx$-resolutions $\gamma'\colon X'\to M'$ and 
$\gamma''\colon X''\to M''$
and that the given sequence
is $\ahom(\catx,-)$-exact. Then $M$
admits a proper $\catx$-resolution $\gamma\colon X\to M$
such that there exists a commutative diagram  whose top row is  degreewise
split exact.  
$$
\xymatrix{
0\ar[r] & 
X' \ar[r]\ar[d]_{\gamma'} 
& X \ar[r]\ar[d]_{\gamma} 
& X'' \ar[r]\ar[d]_{\gamma''} & 0 \\
0\ar[r] & M' \ar[r] & M \ar[r] & M'' \ar[r] & 0
}
$$
\item \label{horseshoe01item2}
Assume that $M'$ and $M''$ 
admit proper $\caty$-coresolutions $\delta'\colon M'\to Y'$ and 
$\delta''\colon M''\to Y''$
and that the given sequence
is $\ahom(-,\caty)$-exact. Then $M$ admits a
proper $\caty$-coresolution $\delta\colon M\to Y$
such that there exists a commutative diagram whose bottom row is  degreewise
split exact.
$$
\xymatrix{
&&&
0\ar[r] & M' \ar[r]\ar[d]_{\delta'} & M \ar[r]\ar[d]_{\delta} & M'' \ar[r]\ar[d]_{\delta''} & 0 \\
&&&
0\ar[r] 
& Y' \ar[r]
& Y \ar[r]
& Y'' \ar[r] & 0 &\hspace{6.5mm}& \qed
}
$$
\end{enumerate}
\end{lem}

The final result of this section is for Corollary~\ref{balance03}.
It follows from~\cite[(2.3)]{sather:sgc}.

\begin{lem} \label{cogen01}
For each integer $n\geq 0$, let $\catx_n$ and $\caty_n$ be subcategories of $\cata$
such that $\catx_n$ and $\caty_n$ are closed under extensions when $n\geq 2$.
\begin{enumerate}[\quad\rm(a)]
\item \label{cogen01item1}
If $\catx_n$ is a cogenerator for $\catx_{n+1}$ for each $n\geq 0$ and $\catx_n\perp\catx_0$
for each $n\geq 1$, then $\catx_n$ is an injective cogenerator for $\catx_{n+j}$ for each $n,j\geq 0$.\item \label{cogen01item2}
If $\caty_n$ is a generator for $\caty_{n+1}$ for each $n\geq 0$ and $\caty_0\perp\caty_n$
for each $n\geq 1$, then $\caty_n$ is a projective generator for $\caty_{n+j}$ for each $n,j\geq 0$.
\qed
\end{enumerate}
\end{lem}

\section{Categories of Interest} \label{sec8}

Much of the motivation for this work comes from module categories.
In reading this paper, the reader may find it helpful to keep in mind the examples of this
section, wherein $R$ is a commutative ring.  
We return to these examples explicitly in Sections~\ref{sec3} and~\ref{sec2}.

\begin{notation} \label{notation08}
Let $\catm(R)$ denote the category of $R$-modules.
For simplicity, we write $\catp(R)=\catp(\catm(R))$ 
and $\cati(R)=\cati(\catm(R))$.
Also set $\catabel=\catm(\mathbb{Z})$, the category
of abelian groups.
If $\catx(R)$ is a subcategory of $\catm(R)$, then $\catx^f(R)$ is the  subcategory
of finitely generated modules in $\catx(R)$.  
\end{notation}

The study of semidualizing modules was initiated independently (with different names)
by Foxby~\cite{foxby:gmarm}, Golod~\cite{golod:gdagpi},
and Vasconcelos~\cite{vasconcelos:dtmc}.

\begin{defn} \label{notation08a}
An $R$-module $C$ is \emph{semidualizing} if it satisfies the following:
\begin{enumerate}[\quad(1)]
\item $C$ admits a (possibly unbounded) resolution by finite rank free $R$-modules,
\item the natural homothety map $R\to\Hom_R(C,C)$ is an isomorphism, and
\item $\ext_R^{\geq 1}(C,C)=0$.
\end{enumerate}
A finitely generated projective  $R$-module of rank 1 is semidualizing.
If $R$ is Cohen-Macaulay, then $D$  is \emph{dualizing}
if it is semidualizing and $\id_R(D)$ is finite.
\end{defn}

Based on the work of Enochs and Jenda~\cite{enochs:gipm},
the following notions were introduced and studied in this generality by
Holm and J\o rgensen~\cite{holm:smarghd}
and White~\cite{white:gpdrsm}.

\begin{notationdefinition} \label{notation08b}
Let $C$ be a semidualizing $R$-module.
An $R$-module is \emph{$C$-projective}
(resp.,  \emph{$C$-injective})
if it is isomorphic to  $P\otimes_R C$ for some projective 
$R$-module $P$
(resp., $\Hom_R(C,I)$ for some injective $R$-module $I$).
The categories of 
$C$-projective
and $C$-injective $R$-modules
are denoted
$\catpc(R)$ and $\catic(R)$, respectively.

A \emph{complete $\catp\catpc$-resolution} is a complex $X$ of $R$-modules 
satisfying the following:
\begin{enumerate}[\quad(1)]
\item $X$ is exact and $\Hom_R(-,\catpc(R))$-exact, and
\item $X_n$ is projective when  $n\geq 0$ and $X_n$ is  $C$-projective when $n< 0$.
\end{enumerate}
An $R$-module $G$ is \emph{$\text{G}_C$-projective} if there
exists a complete $\catp\catpc$-resolution $X$ such that $G\cong\coker(\partial^X_1)$,
in which case $X$ is a \emph{complete $\catp\catpc$-resolution of $G$}.  We let
$\catgpc(R)$ denote the subcategory of $\text{G}_C$-projective $R$-modules.

The terms \emph{complete $\catic\cati$-coresolution}
and \emph{$\text{G}_C$-injective} are defined dually, and
$\catgic(R)$ is the subcategory of $\text{G}_C$-injective $R$-modules.
\end{notationdefinition}

\begin{fact} \label{ICPG}
Let $C$ be a semidualizing $R$-module.
One has $\catp(R)\cup\catpc(R)\subseteq\catgpc(R)$,
and $\catpc(R)$ is an injective cogenerator for 
$\catgpc(R)$ by~\cite[(3.2),(3.6),(3.9)]{white:gpdrsm}. 
Dually, one has $\cati(R)\cup\catic(R)\subseteq\catgic(R)$,
and $\catic(R)$ is a projective generator for 
$\catgic(R)$.
\end{fact}

The next definition was first introduced by Auslander and 
Bridger~\cite{auslander:adgeteac,auslander:smt}
in the case $C=R$, and in this generality
by Golod~\cite{golod:gdagpi} and Vasconcelos~\cite{vasconcelos:dtmc}.

\begin{notationdefinition} \label{notation08c}
Assume that $R$ is noetherian,  and let $C$ be a
semidualizing $R$-module.  A finitely generated $R$-module
$H$ is \emph{totally $C$-reflexive} if 
\begin{enumerate}[\quad(1)]
\item $\ext_R^{\geq 1}(H,C)=0=\ext_R^{\geq 1}(\Hom_R(H,C),C)$, and 
\item the natural biduality map $H\to\Hom_R(\Hom_R(H,C),C)$ is an isomorphism.
\end{enumerate}
Let
$\catgc(R)$ denote the subcategory of totally $C$-reflexive $R$-modules.
\end{notationdefinition}

\begin{fact} \label{gc}
Assume that $R$ is noetherian  and let $C$ be a
semidualizing $R$-module.  
One has $\catgc(R)=\catgpc^f(R)$ 
by~\cite[(5.4)]{white:gpdrsm}, and so
$\catp^f(R)\cup\catpc^f(R)\subseteq\catgc(R)$.
Also, $\catpc^f(R)$ is an injective cogenerator for 
$\catgc(R)$ by~\cite[(3.9),(5.3),(5.4)]{white:gpdrsm}.  
\end{fact}

Over a noetherian ring, the next categories were introduced by 
Avramov and Foxby~\cite{avramov:rhafgd}
when $C$ is dualizing, and 
by Christensen~\cite{christensen:scatac} for arbitrary $C$.\footnote{Note 
that these works (and others) use the notation $\catac(R)$
and $\catbc(R)$ for certain categories of complexes, while our 
categories consist precisely of the modules in these categories
by~\cite[(4.10)]{christensen:scatac}.}
In the non-noetherian setting, these definitions are from~\cite{holm:fear,white:gpdrsm}.

\begin{notationdefinition} \label{notation08d}
Let $C$ be a
semidualizing $R$-module.  
The \emph{Auslander class} of $C$ is the subcategory $\catac(R)$
of $R$-modules $M$ such that 
\begin{enumerate}[\quad(1)]
\item $\tor^R_{\geq 1}(C,M)=0=\ext_R^{\geq 1}(C,C\otimes_R M)$, and
\item The natural map $M\to\Hom_R(C,C\otimes_R M)$ is an isomorphism.
\end{enumerate}
The \emph{Bass class} of $C$ is the subcategory $\catbc(R)$
of $R$-modules $N$ such that 
\begin{enumerate}[\quad(1)]
\item $\ext_R^{\geq 1}(C,N)=0=\tor^R_{\geq 1}(C,\Hom_R(C,N))$, and 
\item The natural evaluation map $C\otimes_R\Hom_R(C,N)\to N$ is an isomorphism.
\end{enumerate}
\end{notationdefinition}

\begin{fact}\label{projac}
Let $C$ be a
semidualizing $R$-module.  
If any two $R$-modules in a short exact
sequence are in $\catac(R)$, respectively $\catbc(R)$, then so is the
third; see~\cite[(6.7)]{holm:fear}.
There are containments 
$\finrescatpr\cup\fincorescaticr \subseteq\catac(R)\subseteq
\propcorescatic$
and
$\finrescatpcr\cup\fincorescatir \subseteq\catbc(R)\subseteq
\proprescatpc$
by~\cite[(6.4),(6.6)]{holm:fear}
and~\cite[(2.4)]{takahashi:hasm}.
\end{fact}

\section{Strict and Proper Resolutions} \label{sec9}

This section focuses on the existence of certain proper resolutions which,
following~\cite{avramov:aratc}, we call ``strict''.
Our treatment 
focuses on the use of ``approximations'' (special cases of the
``special precovers'' of~\cite{enochs:rha}) and blends the approaches 
of~\cite{auslander:htmcma},
\cite{avramov:aratc},
and~\cite{enochs:rha}.

\begin{defn}\label{ab01}
Fix an object $M$ in $\cata$. 
A \emph{bounded strict $\catw\catx$-resolution of $M$}
is a bounded $\catx$-resolution $X\xra{\simeq} M$ such that
$X_n$ is an object in $\catw$ for each $n\geq 1$.  
An exact sequence in $\cata$
$$0\to K \to X_0 \to M \to 0$$
such that $K\in\finrescatw$ and
$X_0\in \catx$ is called an \emph{$\catw\catx$-approximation of $M$}.
The term \emph{$\catw\catx$-hull of $M$} is used for
an exact sequence in $\cata$
$$0\to M\to K'\to X'\to 0$$
such that $K'\in\finrescatw$ and
$X'\in\catx$.
The terms \emph{bounded strict $\caty\catv$-coresolution},
\emph{$\caty\catv$-coapproximation}
and \emph{$\caty\catv$-cohull}
are defined dually.
\end{defn}

The first result of this section outlines the properness properties of 
certain (co)resolutions and (co)approximations.

\begin{lem} \label{xhat}
Assume $\catx\perp\catw$ and $\catv\perp\caty$.
\begin{enumerate}[\quad\rm(a)]
\item \label{xhatitem01}
Bounded $\catw$-resolutions are $\catx$-proper and hence $\catw$-proper.
\item \label{xhatitem02}
If $\catw$ is an injective cogenerator for $\catx$, then 
bounded strict $\catw\catx$-resolutions
are $\catx$-proper and $\catw\catx$-approximations are
$\ahom(\catx,-)$-exact.
\item \label{xhatitem03}
Bounded $\catv$-coresolutions are $\caty$-proper and hence $\catv$-proper.
\item \label{xhatitem04}
If $\catv$ is a projective generator for $\caty$, then 
bounded strict $\caty\catv$-coresolutions
are $\caty$-proper and $\caty\catv$-coapproximations are
$\ahom(-,\caty)$-exact.
\end{enumerate}
\end{lem}

\begin{proof}
We prove parts~\eqref{xhatitem01} and~\eqref{xhatitem02};  
the others are proved dually.

\eqref{xhatitem01}
Let $M$ be an object in $\cata$ 
admitting a bounded $\catw$-resolution $W\to M$.
We need to show that $\ahom(X,W^+)$ is exact for each
object $X$ in $\catx$.
Set $M_n=\coker(\partial^W_{n+2})$ and, when $n\geq 0$, consider the 
associated exact sequence 
\begin{equation*} 
0\to M_n\to W_{n}\to M_{n-1}\to 0.
\end{equation*}
The object $M_n$ is in $\finrescatw$ for each $n$.
Lemma~\ref{gencat01} implies 
$\catx\perp\finrescatw$, and so the displayed sequence is $\ahom(\catx,-)$-exact
by Lemma~\ref{perp03}\eqref{perp03item2}.
It follows that  $W^+$ is $\ahom(\catx,-)$-exact as well,
that is, the resolution is $\catx$-proper.

\eqref{xhatitem02}
Let $X\to M$ be a bounded strict $\catw\catx$-resolution
such that $X_i=0$ for each $i>n$, and set
$K=\im(\partial^X_1)$.  
The next exact sequence is a bounded $\catw$-resolution
\begin{equation} \label{exact100}
0\to X_n\to\cdots\to X_1\to K\to 0
\end{equation}
and so part~\eqref{xhatitem01}
implies that it is $\ahom(\catx,-)$-exact.
The following sequence
\begin{equation} \label{exact101}
0\to K\to X_0\to M\to 0
\end{equation}
is a $\catw\catx$-approximation.
Once we show that $\catw\catx$-approximations are
$\ahom(\catx,-)$-exact, we will conclude that 
$X$ is $\catx$-proper by splicing the sequences~\eqref{exact100}
and~\eqref{exact101}.

Consider a $\catw\catx$-approximation
as in~\eqref{exact101}.
Using Lemma~\ref{gencat01},
the assumption $\catx\perp\catw$ implies $\catx\perp K$.  
Thus, for each $X'\in\catx$
the long exact sequence in $\aext(X',-)$
associated to~\eqref{exact101} implies that~\eqref{exact101} is $\ahom(\catx,-)$-exact.
\end{proof}

The next two lemmata provide useful conditions guaranteeing the existence of 
proper (co)resolutions. Lemma~\ref{contain01} is for 
use in Proposition~\ref{detect01a}.  

\begin{lem} \label{xhatcor}
Assume that $\catx$ and $\caty$ are closed under extensions, 
$\catw$ is a cogenerator
for $\catx$, and $\catv$ is a generator
for $\caty$. Let $M$ and $N$ be objects in $\cata$.
\begin{enumerate}[\quad\rm(a)]
\item \label{xhatitem05'}
If $\xpd(M)<\infty$, then $M$
has  a
$\catw\catx$-approximation,
a $\catw\catx$-hull, and a bounded strict $\catw\catx$-resolution 
$X\xra{\simeq} M$ such that $X_i=0$ for $i>\xpd(M)$.
\item \label{xhatitem05}
If $\catw$ is an injective cogenerator for 
$\catx$, then $\finrescatx$ is a subcategory of $\proprescatx$.
\item \label{xhatitem06'}
If $\yid(N)<\infty$, then $N$
has  a
$\caty\catv$-coapproximation, a
$\caty\catv$-cohull,
and a bounded strict $\caty\catv$-coresolution $N\xra{\simeq}Y$ such that
$Y_{-i}=0$ for $i>\yid(N)$.
\item \label{xhatitem06}
If $\catv$ is a projective generator for $\caty$, then
$\fincorescaty$ is a subcategory of $\propcorescaty$.
\end{enumerate}
\end{lem}

\begin{proof}
Parts~\eqref{xhatitem05'} and~\eqref{xhatitem06'} follow 
as in~\cite[(1.1)]{auslander:htmcma}.
Parts~\eqref{xhatitem05} and~\eqref{xhatitem06}
follow from~\eqref{xhatitem05'} and~\eqref{xhatitem06'} using
Lemma~\ref{xhat}\eqref{xhatitem02} and~\eqref{xhatitem04}.
\end{proof}

\begin{lem} \label{contain01}
Assume that $\catx$ and $\caty$ are closed under extensions, 
$\catw$ is a cogenerator
for $\catx$, and $\catv$ is a generator
for $\caty$. \begin{enumerate}[\quad\rm(a)]
\item \label{contain01item1}
If $\catx$ is a subcategory of $\proprescatw$, then $\finrescatx$ is a subcategory of $\proprescatw$.
\item \label{contain01item2}
If $\caty$ is a subcategory of $\propcorescatv$, then $\fincorescaty$ is a subcategory of 
$\propcorescatv$.
\end{enumerate}
\end{lem}

\begin{proof}
We prove part~\eqref{contain01item1}; the proof of part~\eqref{contain01item2} is dual.
Let $M$ be an object in $\finrescatx$.  
By Lemma~\ref{xhatcor}\eqref{xhatitem05'}, the object $M$
admits a $\catw\catx$-approximation
\begin{equation} \label{exseq04}
0\to K\to X\to M\to 0.
\end{equation}
Since
$\catx$ is a subcategory of $\proprescatw$, the object $X$ admits
a proper $\catw$-resolution $W\xra{\simeq}X$.  Set $X'=\im(\partial_1^W)$.
Notice that the object $X'$ is  in $\proprescatw$
and the following natural exact sequence is $\ahom(\catw,-)$-exact
\begin{equation} \label{exseq05}
0\to X'\to W_0\xra{\tau} X\to 0.
\end{equation}
In the following pullback diagram, 
each row and column is exact, 
the bottom row
is~\eqref{exseq04}, and the middle column is~\eqref{exseq05}.
\begin{equation} \label{diag01}
\begin{split}
\xymatrix{
& 0 \ar[d] & 0 \ar[d] \\
& X'\ar[d]\ar[r]^{\cong} & X'\ar[d] \\
0 \ar[r] & U \ar[r]\ar[d]\ar@{}[rd]|<<{\ulcorner}  & W_0\ar[r]^{\pi}\ar[d]_{\tau} & M \ar[r]\ar[d]_{\cong} & 0 \\
0 \ar[r] & K \ar[r]\ar[d] & X\ar[r]\ar[d] & M \ar[r] & 0 \\
& 0 & 0
}
\end{split}
\end{equation}
We will show that $U$ is  in $\proprescatw$ and that the middle row
of~\eqref{diag01} is $\ahom(\catw,-)$-exact.  
It is then straightforward to see that a proper
$\catw$-resolution of $M$ can be obtained by splicing a a proper
$\catw$-resolution of $U$ with the middle row
of~\eqref{diag01}.

Let $W'$ be an object in $\catw$.  
The assumption $\catx\perp\catw$ implies $\catw\perp\catw$ and
so $\aext^1(W',W_0)=0$.  The long exact sequence in $\aext(W',-)$ 
associated to the middle column of~\eqref{diag01} includes 
the next exact sequence
\begin{equation*}
\ahom(W',W_0)\xra{\ahom(W',\tau)}
\ahom(W',X)\to\aext^1(W',X')\to 0.
\end{equation*}
The middle column of~\eqref{diag01} is 
$\ahom(\catw,-)$-exact, so the map $\ahom(W',\tau)$ is surjective, 
and it follows that $\aext^1(W',X')=0$.
Lemma~\ref{perp03}\eqref{perp03item2} implies that
the leftmost column of~\eqref{diag01} is 
$\ahom(W',-)$-exact.  
Since $W'$ is an arbitrary object of $\catw$, this column
is $\ahom(\catw,-)$-exact.  The object $K$ is in $\proprescatw$ by 
Lemma~\ref{xhat}\eqref{xhatitem01}.  Since $X'$ is also an object in
$\proprescatw$, we may apply Lemma~\ref{horseshoe01}\eqref{horseshoe01item1} 
to the leftmost column of~\eqref{diag01} to conclude that
$U$ is  in $\proprescatw$.

To conclude, we need to show that the middle row of~\eqref{diag01}
is $\ahom(W',-)$-exact, that is,
that $\ahom(W',\pi)$ is surjective.  
Applying $\ahom(W',-)$ to the middle and lower rows of~\eqref{diag01}
yields the next commutative diagram with exact rows.
\begin{equation*} 
\xymatrix{
0 \ar[r] & \ahom(W',U) \ar[r]\ar[d] & \ahom(W',W_0)\ar[rr]^{\ahom(W',\pi)}\ar[d]_{\ahom(W',\tau)} 
&& \ahom(W',M) \ar[d]_{\cong}  \\
0 \ar[r] & \ahom(W',K) \ar[r] & \ahom(W',X)\ar[rr] && \ahom(W',M) \ar[r] & 0
}
\end{equation*}
Recalling that $\ahom(W',\tau)$ is surjective, chase this last  diagram to conclude that
$\ahom(W',\pi)$ is also surjective.
\end{proof}

\section{Relative Cohomology} \label{sec5}

This section contains the foundations of our 
relative cohomology theories based on the context of Section~\ref{sec9}.

\begin{notationdefinition} \label{rel02}
Let $M,M',N,N'$ be objects in $\cata$ 
equipped with homomorphisms
$f\colon M\to M'$ and $g\colon N\to N'$.
Assume that $M$ admits a proper $\catx$-resolution 
$\gamma \colon X\to M$, and define
the \emph{$n$th relative $\catx\cata$ cohomology group} as
$$\xaext^n(M,N)=\HH_{-n}(\ahom(X,N))$$
for each integer $n$.
If $M'$ also admits a proper $\catx$-resolution 
$\gamma'\colon X'\to M'$, then 
let $\ol{f}\colon X\to X'$ be
a morphism such that $\gamma'\ol{f}=f\gamma$,
as in Lemma~\ref{rel01}\eqref{rel01item1}, and define
\begin{gather*}
\xaext^n(f,N)=\HH_{-n}(\ahom(\ol{f},N))\colon\xaext^n(M',N)\to\xaext^n(M,N) \\
\xaext^n(M,g)=\HH_{-n}(\ahom(X,g))\colon\xaext^n(M,N)\to\xaext^n(M,N'). 
\end{gather*}
We write $\xaext^{\geq1}(M,\caty)=0$ if $\xaext^{\geq1}(M,Y)=0$
for each object $Y\in \caty$.  When $\catx\subseteq\proprescatw$,
we write $\waext^{\geq1}(\catx,\caty)=0$ if $\waext^{\geq1}(X,\caty)=0$
for each object $X\in \catx$.

The \emph{$n$th relative $\cata\caty$-cohomology}
$\ayext^n(-,-)$ is defined dually.
\end{notationdefinition}

\begin{disc} \label{rel03}
Definition/Notation~\ref{rel02} describes well-defined bifunctors
\begin{align*}
\xaext^n(-,-)&\colon\proprescatx\times\cata\to \catabel & & &
\ayext^n(-,-)&\colon\cata\times \propcorescaty\to \catabel
\end{align*}
by Lemma~\ref{rel01},
and one checks the following natural equivalences readily.
\begin{gather*}
\xaext^{\geq 1}(\catx,-)=0=\ayext^{\geq 1}(-,\caty) \\
\xaext^0(-,-)\cong\ahom(-,-)|_{\proprescatx\times\cata} \qquad
\ext^n_{\catp\cata}(-,-)\cong\aext^n(-,-)|_{\proprescat{P}\times\cata}\\
\ayext^0(-,-)\cong\ahom(-,-)|_{\cata\times \propcorescaty}  \qquad
\ext^n_{\cata\cati}(-,-)\cong \aext^n(-,-)|_{\cata\times \propcorescat{I}}
\end{gather*}
\end{disc}

Lemma~\ref{horseshoe01} yields the following long exact sequences 
as in~\cite[(8.2.3),(8.2.5)]{enochs:rha}.

\begin{lem} \label{notation06a}
Let $M$ and $N$ be objects in $\cata$, and consider an exact  sequence in $\cata$
$$\mathbf{L}=\quad 0\to L'\xra{f'} L\xra{f} L''\to 0.$$
\begin{enumerate}[\quad\rm(a)]
\item \label{06aitem1}
Assume that the sequence $\mathbf{L}$ is $\ahom(\catx,-)$-exact.
If the object $M$ is in $\proprescatx$, then $\mathbf{L}$ induces a functorial long exact sequence
\begin{align*}
\cdots \to
& \xaext^n(M,L') \xra{\xaext^n(M,f')}
 \xaext^n(M,L) \xra{\xaext^n(M,f)}
  \\
& \xaext^n(M,L'')\xra{\eth^n_{\catx\!\cata}(M,\mathbf{L})}
 \xaext^{n+1}(M,L') \xra{\xaext^{n+1}(M,f')}
 \cdots 
\end{align*}
\item \label{06aitem2}
Assume that the sequence $\mathbf{L}$ is $\ahom(\catx,-)$-exact.
If the objects $L',L,L''$ are in 
$\proprescatx$, then $\mathbf{L}$ induces a functorial long exact sequence
\begin{align*}
\cdots \to
& \xaext^n(L'',N) \xra{\xaext^n(f,N)}
 \xaext^n(L,N) \xra{\xaext^n(f',N)}
  \\
& \xaext^n(L',N)\xra{\eth^n_{\catx\!\cata}(\mathbf{L},N)}
 \xaext^{n+1}(L'',N) \xra{\xaext^{n+1}(f,N)}
 \cdots 
\end{align*}
\item \label{06aitem3}
Assume that the sequence $\mathbf{L}$ is  $\ahom(-,\caty)$-exact.
If the object $N$ is in $\propcorescaty$, then $\mathbf{L}$ induces a functorial long exact sequence
\begin{align*}
\cdots \to
& \ayext^n(L'',N) \xra{\ayext^n(f,N)}
 \ayext^n(L,N) \xra{\ayext^n(f',N)}
  \\
& \ayext^n(L',N)\xra{\eth^n_{\cata\caty}(\mathbf{L},N)}
 \ayext^{n+1}(L'',N) \xra{\ayext^{n+1}(f,N)}
 \cdots 
\end{align*}
\item \label{06aitem4}
Assume that the sequence $\mathbf{L}$ is  $\ahom(-,\caty)$-exact.
If the objects $L',L,L''$ are in 
$\propcorescaty$, then $\mathbf{L}$ induces a functorial long exact sequence
\begin{align*}
&&&&&&&&\cdots \to
& \ayext^n(M,L') \xra{\ayext^n(M,f')}
 \ayext^n(M,L) \xra{\ayext^n(M,f)}
  \\
\hspace{0.5em}&&&&&&&&& \ayext^n(M,L'')\xra{\eth^n_{\cata\caty}(M,\mathbf{L})}
 \ayext^{n+1}(M,A') \xra{\ayext^{n+1}(M,f')}
 \cdots &&&& \qed
\end{align*}
\end{enumerate}
\end{lem}

To prove the next ``dimension-shifting'' lemma, 
comparable to~\cite[(8.2.4),(8.2.6)]{enochs:rha},
use
the long exact sequences from 
Lemma~\ref{notation06a} with
the vanishing from Remark~\ref{rel03}.

\begin{lem} \label{dimshift01}
Let $M$ and $N$ be objects in $\cata$, and consider an exact  sequence in $\cata$
$$\mathbf{L}=\quad 0\to L'\xra{f'} L\xra{f} L''\to 0.$$
\begin{enumerate}[\quad\rm(a)]
\item \label{dimshift01item1}
Assume that the sequence $\mathbf{L}$ is  $\ahom(\catx,-)$-exact
and that $M$ is in $\proprescatx$.
If $\xaext^{\geq 1}(M,L)=0$, e.g., if $M$ is in $\catx$,   then the
following map is an isomorphism for each  $n\geq 1$
$$\eth^n_{\catx\!\cata}(M,\mathbf{L})\colon\xaext^n(M,L'')\xra{\cong}\xaext^{n+1}(M,L').$$
\item \label{dimshift01item2}
Assume that the sequence $\mathbf{L}$ is  $\ahom(\catx,-)$-exact
and that  $L,L',L''$ are in $\proprescatx$.
If $\xaext^{\geq 1}(L,N)=0$, e.g., if $L$ is in $\catx$,    then the
following map is an isomorphism for each  $n\geq 1$
$$\eth^n_{\catx\!\cata}(\mathbf{L},N)\colon\xaext^n(L',N)\xra{\cong}\xaext^{n+1}(L'',N).$$
\item \label{dimshift01item3}
Assume that the sequence $\mathbf{L}$ is  $\ahom(-,\caty)$-exact
and that $N$ is in $\propcorescaty$.
If $\ayext^{\geq 1}(L,N)=0$, e.g., if $N$ is in $\caty$,   then the
following map is an isomorphism for each  $n\geq 1$
$$\eth^n_{\cata\caty}(\mathbf{L},N)\colon\ayext^n(L',N)\xra{\cong}\ayext^{n+1}(L'',N).
$$
\item \label{dimshift01item4}
Assume that the sequence $\mathbf{L}$ is  $\ahom(-,\caty)$-exact
and that $L,L',L''$ are  in $\propcorescaty$.
If $\ayext^{\geq 1}(M,L)=0$, e.g., if $L$ is in $\caty$,   then the
following map is an isomorphism for each  $n\geq 1$
\begin{align*}
\hspace{0.8em}&&&&&&&&&&&&&&
\eth^n_{\cata\caty}(M,\mathbf{L})&\colon\ayext^n(M,L'')\xra{\cong}\ayext^{n+1}(M,L').
&&&&&&&&&&&\qed
\end{align*}
\end{enumerate}
\end{lem}

The next result is motivated by~\cite[(4.2.2.a)]{avramov:aratc}.

\begin{prop} \label{detect00}
Let $M$ and $N$ be objects in $\proprescatx$ and $\propcorescaty$,
respectively, and let
$n$ be a nonnegative integer.
\begin{enumerate}[\quad\rm(a)]
\item \label{detect00item2}
Assume that $\catx$ is closed under direct summands and 
$\xaext^{n+1}(M,-)=0$.
If $X\to M$ is a proper $\catx$-resolution, then 
$\ker(\partial^X_{n-1})\in\catx$ and 
$\xpd(M)\leq n$.
\item \label{detect00item1}
Assume that one of the following conditions holds:
\begin{enumerate}[\quad\rm(1)]
\item \label{detect00item11}
$\catx\perp\catx$, or
\item \label{detect00item12}
$\catx$ is closed under extensions and $\catw$ is an
injective cogenerator for $\catx$.
\end{enumerate}
Then $\xaext^n(M,-)=0$ whenever $n>\xpd(M)$.
\item \label{detect00item4}
Assume that $\caty$ is closed under direct summands
and $\ayext^{n+1}(-,N)=0$.
If $N\to Y$ is a proper $\caty$-coresolution, then 
$\coker(\partial^Y_{1-n})\in\caty$ and 
$\yid(N)\leq n$.
\item \label{detect00item3}
Assume that one of the following conditions holds:
\begin{enumerate}[\quad\rm(1)]
\item \label{detect00item31}
$\caty\perp\caty$, or
\item \label{detect00item32}
$\caty$ is closed under extensions and $\catv$ is a
projective cogenerator for $\caty$.
\end{enumerate}
Then $\ayext^n(-,N)=0$ whenever $n>\yid(N)$.
\end{enumerate}
\end{prop}

\begin{proof}
We prove parts~\eqref{detect00item2} and ~\eqref{detect00item1};  
the proofs of~\eqref{detect00item4} and~\eqref{detect00item3}
are dual.

\eqref{detect00item2}
Let $X\to M$ be a proper $\catx$-resolution, and 
set  $M_j=\coker(\partial^X_{j+2})$ for each integer $j$.
Note $M_j\in\proprescatx$ and 
$M\cong M_{-1}$, and consider the exact sequences
\begin{equation} \label{exact01} \tag{$\ast_j$}
0\to M_j\to X_j\xra{\epsilon_j} M_{j-1}\to 0
\end{equation}
when $j\geq 0$,
which are $\ahom(\catx,-)$-exact.

Assume first $\xaext^{1}(M,-)=0$.
An application of Lemma~\ref{notation06a}\eqref{06aitem1}
to the sequence $(\ast_0)$ yields the following exact sequence
$$0\to\ahom(M,M_0)\to\ahom(M,X_0)\xra{\ahom(M,\epsilon_0)}
\ahom(M,M)\to 0.$$
Hence, there exists
$\phi\in\ahom(M,X_0)$ such that $\epsilon_0\phi=\id_M$.
It follows that $M$ is a direct summand of $X_0$, and so 
$M\in\catx$ because $\catx$ is closed under direct summands.

Now assume $\xaext^{n+1}(M,-)=0$.
Apply Lemma~\ref{dimshift01}\eqref{dimshift01item2} to each
sequence $(\ast_j)$ inductively to conclude $\xaext^{1}(M_{n-1},-)=0$.
The previous paragraph now implies $\ker(\partial_{n-1}^X)=M_{n-1}\in\catx$.
The conclusion
$\xpd(M)\leq n$ is now immediate.

\eqref{detect00item1}
Assume without loss of generality that $p=\xpd(M)$ is finite.
It suffices to show that $M$ admits a proper $\catx$-resolution
$X\to M$ such that $X_n=0$ when $n>p$.
If condition (1) holds, then Lemma~\ref{xhat}\eqref{xhatitem01}
implies that every $\catx$-resolution 
$X\to M$ such that $X_n=0$ for each $n>p$
is proper.
On the other hand, if condition (2) holds, then 
Lemmas~\ref{xhat}\eqref{xhatitem02}
and~\ref{xhatcor}\eqref{xhatitem05'} 
yield the desired conclusion.
\end{proof}

The rest of this section is devoted to the study of the following
comparison maps.

\begin{notationdefinition} \label{rel02a}
Let $M,N$ be objects in $\cata$. 
\begin{enumerate}[\rm(a)]
\item \label{rel02aitem1}
When $M$ admits a proper $\catw$-resolution 
$\gamma \colon W\to M$ and a proper $\catx$-resolution 
$\gamma'\colon X\to M$, let $\ol{\id_M}\colon W\to X$ be a quasiisomorphism
such that $\gamma=\gamma'\ol{\id_M}$,
as in 
Lemma~\ref{rel01}\eqref{rel01item1}, and set
$$
\xwacomp^n(M,N)=\HH_{-n}(\ahom(\ol{\id_M},N))\colon\xaext^n(M,N)\to\waext^n(M,N). 
$$
\item \label{rel02aitem2}
When $M$ admits a projective resolution 
$\gamma \colon P\to M$ and a proper $\catx$-resolution 
$\gamma' \colon X\to M$, let $\wt{\id_M}\colon P\to X$ be a quasiisomorphism
such that $\gamma=\gamma'\wt{\id_M}$,
as in 
Lemma~\ref{rel01}\eqref{rel01item1'}, and set
$$
\xaacomp^n(M,N)=\HH_{-n}(\ahom(\wt{\id_M},N))\colon\xaext^n(M,N)\to\aext^n(M,N). 
$$
\item \label{rel02aitem3}
When $N$ admits a proper $\caty$-coresolution 
$\delta \colon N\to Y$ and a proper $\catv$-coresolution 
$\delta'\colon N\to V$, let $\ol{\id_N}\colon Y\to V$ be a quasiisomorphism
such that $\delta'=\ol{\id_N}\delta$,
as in 
Lemma~\ref{rel01}\eqref{rel01item2}, and set
$$
\ayvcomp^n(M,N)=\HH_{-n}(\ahom(M,\ol{\id_N}))\colon\ayext^n(M,N)\to\avext^n(M,N). 
$$
\item \label{rel02aitem4}
When $N$ admits a proper $\caty$-coresolution 
$\delta \colon N\to Y$ and an injective resolution 
$\delta'\colon N\to I$, let $\wt{\id_N}\colon Y\to I$ be a quasiisomorphism
such that $\delta'=\wt{\id_N}\delta$,
as in 
Lemma~\ref{rel01}\eqref{rel01item2'}, and set
$$
\aaycomp^n(M,N)=\HH_{-n}(\ahom(M,\wt{\id_N}))\colon\ayext^n(M,N)\to\aext^n(M,N). 
$$
\end{enumerate}
\end{notationdefinition}

\begin{disc}
Lemma~\ref{rel01} shows that Definition/Notation~\ref{rel02a} describes
well-defined natural transformations
that are independent of resolutions and liftings.
\begin{align*}
\xwacomp ^n(-,-)&\colon\xaext^n(-,-)|_{(\proprescatw\cap \proprescatx)\times\cata}\to 
\waext^n(-,-)|_{(\proprescatw\cap \proprescatx)\times\cata} \\
\xaacomp ^n(-,-)&\colon\xaext^n(-,-)|_{(\proprescatp\cap \proprescatx)\times\cata}\to 
\aext^n(-,-)|_{(\proprescatp\cap \proprescatx)\times\cata} \\
\ayvcomp ^n(-,-)&\colon\ayext^n(-,-)|_{\cata \times(\propcorescatv\cap \propcorescaty)}\to 
\avext^n(-,-)|_{\cata \times(\propcorescatv\cap \propcorescaty)}\\
\aaycomp ^n(-,-)&\colon\ayext^n(-,-)|_{\cata \times(\propcorescati\cap \propcorescaty)}\to 
\aext^n(-,-)|_{\cata \times(\propcorescati\cap \propcorescaty)}
\end{align*}
\end{disc}

The next result compares to~\cite[(4.2.3)]{avramov:aratc}.

\begin{prop} \label{detect01}
Assume 
$\catx\perp\catw$ and 
$\catv\perp\caty$, and fix objects  $M\in\finrescatw$ and
$N\in\fincorescatv$.
\begin{enumerate}[\quad\rm(a)]
\item \label{detect01item1}
The following natural transormations are isomorphisms
for each $n$
$$\xwacomp^n(M,-)\colon\xaext^n(M,-)\xra{\cong}\waext^n(M,-).$$ 
\item \label{detect01item2}
The following natural transormations are isomorphisms
for each $n$
$$\ayvcomp^n(-,N)\colon\ayext^n(-,N)\xra{\cong}\avext^n(-,N).$$ 
\end{enumerate}
\end{prop}

\begin{proof}
We prove part~\eqref{detect01item1};  the proof of~\eqref{detect01item2}
is dual.

Let $W\to M$ be a bounded $\catw$-resolution.
Lemma~\ref{xhat}\eqref{xhatitem01}
implies that $W$ is  $\catx$-proper and $\catw$-proper,
so $\waext^n(M,-)$ and $\xaext^n(M,-)$ are defined.  Further,
in the notation of Definition~\ref{rel02a}\eqref{rel02aitem1}, we can take
$\ol{\id_M}=\id_W$, and so there are equalities
$$\xwacomp^n(M,-)=\HH_{-n}(\ahom(\ol{\id_M},-))=\HH_{-n}(\ahom(\id_W,-))
=\id_{\HH_{-n}(\ahom(W,-))}$$
which establish the desired result.
\end{proof}

The next lemma is a tool for the proofs of Propositions~\ref{detect01a} 
and~\ref{detect01b}.  Note that we do not assume that the complexes satisfy
any properness conditions.

\begin{lem} \label{quisos}
Let $M$ and $N$ be objects in $\cata$, and assume $\catx\perp\catw$
and $\catv\perp\caty$.
\begin{enumerate}[\quad\rm(a)]
\item \label{quisositem1}
Let  $\alpha\colon X\to X'$ be a quasiisomorphism between
bounded below complexes in $\catx$.
If $\wpd(N)<\infty$,
then the morphism $\ahom(\alpha,N)\colon\ahom(X',N)\to\ahom(X,N)$
is a quasiisomorphism.
\item \label{quisositem2}
Let  $\beta\colon Y\to Y'$ be a quasiisomorphism between
bounded above complexes in $\caty$.
If $\vid(M)<\infty$,
then the morphism $\ahom(M,\beta)\colon\ahom(M,Y)\to\ahom(M,Y')$
is a quasiisomorphism.
\end{enumerate}
\end{lem}

\begin{proof}
We prove part~\eqref{quisositem1}; the proof of part~\eqref{quisositem2} is dual.

It suffices to show that  $\cone(\ahom(\alpha,N))$ is exact.
From the next isomorphism 
$$\cone(\ahom(\alpha,N))\cong\shift\ahom(\cone(\alpha),N)$$
we need to show that $\ahom(\cone(\alpha),N)$ is exact.  Note that
$\cone(\alpha)$ is an exact, bounded below complex in $\catx$.  
Set  $M_j=\ker(\partial^{\cone(\alpha)}_j)$ for each integer $j$, 
and note $M_{j-1}\in\catx$ for $j \ll 0$.
Consider the exact sequences
\begin{equation} \label{exact01'} \tag{$\ast_j$}
0\to M_j\to \cone(\alpha)_j\to M_{j-1}\to 0.
\end{equation}
The condition $\catx\perp\catw$ implies
$\catx\perp N$ by Lemma~\ref{gencat01}.
Hence, induction on $j$ using Lemma~\ref{perp03}\eqref{perp03item1} implies
$\aext^{\geq 1}(M_j,N)=0$ for each $j$ and so each sequence $(\ast_j)$ is
$\ahom(-,N)$-exact.
It follows that $\ahom(\cone(\alpha),N)$ is exact.
\end{proof}

The next two results compare to~\cite[(4.2.4)]{avramov:aratc}.
Note that Lemmas~\ref{xhatcor} and~\ref{contain01} provide conditions
implying $\finrescatx\subseteq\proprescatx\cap\proprescatw$ and 
$\fincorescaty\subseteq\propcorescaty\cap\propcorescatv$.

\begin{prop} \label{detect01a}
Let $M$ and $N$ be objects in $\cata$, and assume $\catx\perp\catw$
and $\catv\perp\caty$.
\begin{enumerate}[\quad\rm(a)]
\item \label{detect01aitem1}
If $M$ is in $\proprescatx\cap\proprescatw$ and 
$N$ is in $\finrescatw$, then 
the following natural map is an isomorphism for each $n$
$$\xwacomp^n(M,N)\colon\xaext^n(M,N)\xra{\cong}\waext^n(M,N).$$ 
\item \label{detect01aitem2}
If $M$ is in $\fincorescatv$ and 
$N$ is in $\propcorescaty\cap\propcorescatv$, then 
the following natural map is an isomorphism for each $n$
$$\ayvcomp^n(M,N)\colon\ayext^n(M,N)\xra{\cong}\avext^n(M,N).$$ 
\end{enumerate}
\end{prop}

\begin{proof}
We prove part~\eqref{detect01aitem1}; the proof of part~\eqref{detect01aitem2}
is dual.  

The object
$M$ has a proper $\catw$-resolution 
$\gamma\colon W\to M$ and a proper $\catx$-resolution 
$\gamma'\colon X\to M$.
Lemma~\ref{rel01}\eqref{rel01item1}
yields a quasiisomorphism
$\ol{\id_M}\colon W\to X$ 
such that $\gamma=\gamma'\ol{\id_M}$,
and Lemma~\ref{quisos}\eqref{quisositem1} implies that the
morphism $\ahom(\ol{\id_M},N)$ is a quasiisomorphism.
The result now follows from the definition of
$\xwacomp^n(M,N)$.
\end{proof}

\begin{prop} \label{detect01b}
Let $M$ and $N$ be objects in $\cata$, and assume $\catx\perp\catw$
and $\catv\perp\caty$.
\begin{enumerate}[\quad\rm(a)]
\item \label{detect01bitem1}
If $M$ is in $\proprescatx\cap\proprescatp$ and 
$N$ is in $\finrescatw$, then 
the following natural map is an isomorphism for each $n$
$$\xaacomp^n(M,N)\colon\xaext^n(M,N)\xra{\cong}\aext^n(M,N).$$ 
\item \label{detect01bitem2}
If $M$ is in $\fincorescatv$ and 
$N$ is in $\propcorescaty\cap\propcorescati$, then 
the following natural map is an isomorphism for each $n$
$$\aaycomp^n(M,N)\colon\ayext^n(M,N)\xra{\cong}\aext^n(M,N).$$ 
\end{enumerate}
\end{prop}

\begin{proof}
Argue as in the proof of Proposition~\ref{detect01a}.
When invoking Lemma~\ref{quisos}\eqref{quisositem1}, use the category
$\catx\oplus\catp$ whose objects are precisely 
those of the of the form $X\oplus P$
for some $X\in \catx$ and $P\in\catp$.
\end{proof}

The next two lemmata are tools for Proposition~\ref{detect01a'}
and Theorem~\ref{balance02}.  

\begin{lem} \label{balance01}
Let $\catw$ be a cogenerator for $\catx$ and 
let $\catv$ be a generator for $\caty$.  
\begin{enumerate}[\quad\rm(a)]
\item \label{balance01item1}
If $\catw\perp(\catw\cup\caty)$ and 
$\waext^{\geq 1}(\finrescatw,\catv)=0$,
then $\waext^{\geq 1}(\finrescatw,\caty)=0$.
\item \label{balance01item2}
If $(\catx\cup\catv)\perp\catv$ and 
$\avext^{\geq 1}(\catw,\fincorescatv)=0$,
then $\avext^{\geq 1}(\catx,\fincorescatv)=0$.
\end{enumerate}
\end{lem}

\begin{proof}
We prove part~\eqref{balance01item1};  part~\eqref{balance01item2} is proved dually.  
Fix objects $M$ in $\finrescatw$ and $Y$ in $\caty$, and set $Y_0=Y$.
Because $\catv$ is a generator for $\caty$ there exist exact sequences
$$0\to Y_{n+1}\to V_n\to Y_n\to0$$
with $V_n$ in $\catv$ and $Y_{n+1}$ in $\caty$.  The assumption
$\catw\perp\caty$ implies that each of these sequences is $\ahom(\catw,-)$-exact
by Lemma~\ref{perp03}\eqref{perp03item2}.
Fix an integer $j\geq 1$ and set $p=\wpd(M)$.
The vanishing hypothesis implies $\waext^{\geq 1}(M,V_n)=0$ for each $n$,
and so Lemma~\ref{dimshift01}\eqref{dimshift01item1} inductively
yields the isomorphism in the following sequence
$$\waext^j(M,Y)=\waext^j(M,Y_0)\cong\waext^{j+p}(M,Y_p)=0 $$
where the last equality is from Proposition~\ref{detect00}\eqref{detect00item1}
because $\catw\perp\catw$.
\end{proof}

\begin{lem} \label{quisos'}
Assume 
that $\catw$ is a cogenerator for $\catx$ and
$\catv$ is a generator for $\caty$.
Let $M$ and $N$ be objects in $\cata$
with $\wpd(M)<\infty$ and
$\vid(N)<\infty$.
\begin{enumerate}[\quad\rm(a)]
\item \label{quisos'item1}
Assume $(\catx\cup\catv)\perp\catv$ and 
$\avext^{\geq 1}(\catw,\fincorescatv)=0$.
If  $\alpha\colon X\xra{\simeq} X'$ is a quasiisomorphism between
bounded below complexes in $\catx$,
then the morphism $\ahom(\alpha,N)\colon\ahom(X',N)\to\ahom(X,N)$
is a quasiisomorphism.
\item \label{quisos'item2}
Assume $\catw\perp(\catw\cup\caty)$ and 
$\waext^{\geq 1}(\finrescatw,\catv)=0$.
If  $\beta\colon Y\xra{\simeq} Y'$ is a quasiisomorphism between
bounded above complexes in $\caty$, then
the morphism $\ahom(M,\beta)\colon\ahom(M,Y)\to\ahom(M,Y')$
is a quasiisomorphism.
\end{enumerate}
\end{lem}

\begin{proof}
We prove part~\eqref{quisos'item1}; the proof of part~\eqref{quisos'item2} is dual.

Set $M_j=\ker(\partial_j^{\cone(\alpha)})$ for each $j$, and note 
$M_j\in\catx$ for $j\ll 0$.
As in the proof of Lemma~\ref{quisos},
it suffices to show that each of the following exact sequences
\begin{equation} \label{exacts} \tag{$\ast_j$}
0\to M_j\to \cone(\alpha)_j\to M_{j-1}\to 0.
\end{equation}
is $\ahom(-,N)$-exact.  
The condition $\catx\perp\catv$ implies $M_j\perp\catv$
for $j\ll 0$ and $\cone(\alpha)_j\perp\catv$
for all $j\in\zz$.
Applying Lemma~\ref{perp03}\eqref{perp03item1}
to the sequences $(\ast_j)$
inductively implies $M_j\perp\catv$ for all $j\in\zz$ and so each 
$(\ast_j)$ is $\ahom(-,\catv)$-exact.

Lemma~\ref{balance01}\eqref{balance01item2} implies
$\avext^{\geq 1}(M_j,N)=0$ for $j\ll 0$ and 
$\avext^{\geq 1}(\cone(\alpha)_j,N)=0$ for all $j\in\zz$.
Applying
Lemma~\ref{dimshift01}\eqref{dimshift01item3} to $(\ast_j)$ inductively
yields $\avext^{\geq 1}(M_j,N)=0$
for all $n\in\zz$.  Thus, each sequence $(\ast_j)$ is $\ahom(-,N)$-exact,
as desired.
\end{proof}

The next result  is proved like Proposition~\ref{detect01a}, 
using Lemma~\ref{quisos'} in place of~\ref{quisos}.

\begin{prop} \label{detect01a'}
Assume 
that $\catw$ is a cogenerator for $\catx$ and
$\catv$ is a generator for $\caty$.
Let $M$ and $N$ be objects in $\cata$.
\begin{enumerate}[\quad\rm(a)]
\item \label{detect01a'item1}
Assume $(\catx\cup\catv)\perp\catv$ and 
$\avext^{\geq 1}(\catw,\fincorescatv)=0$.
If $M$ is in $\proprescatx\cap\proprescatw$ and 
$N$ is in $\fincorescatv$, then 
the following map 
is an isomorphism for each $n$
$$\xwacomp^n(M,N)\colon\xaext^n(M,N)\xra{\cong}\waext^n(M,N).$$ 
\item \label{detect01a'item2}
Assume $\catw\perp(\catw\cup\caty)$ and 
$\waext^{\geq 1}(\finrescatw,\catv)=0$.
If $M$ is in $\finrescatw$ and 
$N$ is in $\propcorescaty\cap\propcorescatv$, then 
the following map 
is an isomorphism for each $n$
\begin{align*}
&&&&&&&&&&\hspace{3mm}&&&&\ayvcomp^n(M,N)\colon\ayext^n(M,N)\xra{\cong}\avext^n(M,N).
&&&&&&&&&\hspace{2mm}&&\qed 
\end{align*}
\end{enumerate}
\end{prop}

\section{Relative Perfection} \label{sec3}

This section is concerned with a relative notion of perfection
akin to the Gorenstein perfection of~\cite{avramov:aratc},
the quasi-perfection of~\cite{foxby:qpmcmr}
and the generalized perfection of~\cite{golod:gdagpi}.
We begin with the relevant definitions.

\begin{defn} \label{tiltdef1}
Let $\catao$ be another abelian category with subcategory
$\catxo$ and let $T$ and $T^o$ be objects in $\catx$ and $\catxo$,
respectively.  The pair $(T,T^o)$ is a \emph{relative cotilting pair}
for the quadruple $(\cata,\catx,\catao,\catxo)$ when
the next conditions are satisfied:
\begin{enumerate}
\item The functor $\ahom(-,T)$ maps $\cata$ to $\catao$ and $\catx$ to $\catxo$.
\item The functor $\aohom(-,T^o)$ maps $\catao$ to $\cata$ and $\catxo$ to $\catx$.
\item 
There are natural isomorphisms 
$\aohom(\ahom(-,T),T^o)|_{\catx}\cong\id_{\catx}$ 
and $\ahom(\aohom(-,T^o),T)|_{\catxo}\cong\id_{\catxo}$.
\end{enumerate}
The term \emph{relative tilting pair}
is defined dually.
\end{defn}

\begin{defn} \label{perfectdef1}
Let $T$ be an object in $\cata$. An object $M$ in $\cata$ with
$g=\xpd(M)<\infty$ is \emph{$\catx T$-perfect of grade $g$} if $\aext^n(M,T)=0$
for each $n\neq g$.  The term \emph{$T\caty$-coperfect of cograde $g$}
is defined dually.
\end{defn}

Our motivating example comes from our categories of interest.

\begin{ex} \label{tiltex1}
If $R$ is noetherian and $C$ is a semidualizing $R$-module, then
the pair $(C,C)$ is a relative cotilting pair for 
$(\catm(R),\catgc(R),\catm(R),\catgc(R))$.\footnote{More generally, 
one may take $C$ to be a semidualizing $RS$-bimodule
as in~\cite{holm:fear} and conclude that the pair
$(_RC,C_S)$ is a relative cotilting pair for 
$(\catm(R),\catgc(R),\catm(S^o),\catgc(S^o))$.}
In this case, we write ``$\catgc$-perfect'' instead of 
``$\catgc(R)C$-perfect''.
The class of $\catgc$-perfect $R$-modules includes
the totally $C$-reflexive $R$-modules and the perfect $R$-modules.
When $C=R$, this notion recovers the G-perfect modules of~\cite[Sec.~6]{avramov:aratc}.
\end{ex}

Our main result on relative perfection establishes a duality between 
categories of relatively perfect objects.

\begin{prop} \label{perfect1}
Let $M$ be an object in $\cata$,
and let $\catao$ be an abelian category with subcategories $\catxo$ and $\catyo$.
\begin{enumerate}[\quad\rm(a)]
\item \label{perfect1item1}
Let $(T,T^o)$ be a relative cotilting pair for $(\cata,\catx,\catao,\catxo)$ such that
$\catx\perp T$ and $\catxo\perp T^o$.  Assume that $\cata$ and $\catao$ have
enough projectives.  If $M$ is $\catx T$-perfect of grade $g$, then 
$\aext^g(M,T)$ is an object of $\catao$ that is $\catxo T^o$-perfect of
grade $g$, and 
$\aoext^g(\aext^g(M,T),T^o)\cong M$.
\item \label{perfect1item2}
Let $(U,U^o)$ be a relative tilting pair for $(\cata,\caty,\catao,\catyo)$ such that
$U\perp \caty$ and $U^o\perp \catyo$, and assume that $\cata$ and $\catao$ have
enough injectives.  If $M$ is $U\caty$-coperfect of cograde $g$, then 
$\aext^g(U,M)$ is an object of $\catao$ that is $U^o\catyo$-coperfect of
cograde $g$, and $\aoext^g(U^o,\aext^g(U,M))\cong M$.
\end{enumerate}
\end{prop}

\begin{proof}
We prove part~\eqref{perfect1item1}; the proof of part~\eqref{perfect1item2}
is dual.  

The result is trivial if $M=0$, so assume $M\neq 0$.
Let $X\xra{\simeq}M$ be an $\catx$-resolution such that $X_n=0$ for each
$n>g=\xpd(M)$.  
By assumption, the complex $\ahom(X,T)$ consists of objects and morphisms
in $\catxo$.

As in the proof
of Proposition~\ref{detect01b}, Lemma~\ref{quisos}\eqref{quisositem1}
yields an isomorphism
$$\HH_{-n}(\ahom(X,T))\cong\aext^n(M,T)$$
for each $n$.  
Because $M$ is $\catx T$-perfect of grade $g$, 
we conclude that the complex $\shift^g\ahom(X,T)$
is an $\catxo$-resolution of $\aext^g(M,T)$ such that
$(\shift^g\ahom(X,T))_n=0$ for each $n>g$.  In particular,
the object $\aext^g(M,T)\cong\coker(\ahom(\partial^X_g,T))$
is in $\catao$ and $g^o=\xopd(\aext^g(M,T))\leq g<\infty$.

Similarly, we conclude that there is an isomorphism
$$\HH_{g-n}(\aohom(\ahom(X,T),T^o))\cong\aoext^n(\aext^g(M,T),T^o)$$
for each $n$.  Our assumptions yield the isomorphism in the next sequence
$$\aohom(\ahom(X,T),T^o)\cong X\simeq M$$
while the quasiisomorphism is by construction.  These displays imply
$$\aoext^n(\aext^g(M,T),T^o)\cong
\begin{cases} 0 & \text{if $n\neq g$} \\
M & \text{if $n= g$.} \end{cases}
$$ 
It remains to justify the equality $g^0= g$.
We already know $g^o\leq g$, so suppose
$g^o< g$.  Using Lemma~\ref{quisos}\eqref{quisositem1}
as above,
this would imply 
$\aoext^n(\aext^g(M,T),T^o)=0$ for each $n\geq g$.
In particular, we would have a contradiction from the next sequence
$$\hspace{37.5mm}
0=\aoext^g(\aext^g(M,T),T^o)\cong M. 
\hspace{30mm}\qedhere$$
\end{proof}

We conclude this section with the special case of Proposition~\ref{perfect1}
for our categories of interest.  The special case $C=R$ recovers~\cite[(6.3.1,2)]{avramov:aratc}.

\begin{cor} \label{perfect2}
Let $R$ be a commutative noetherian ring and $C,M$ finitely generated $R$-modules
with $C$ semidualizing and $\gkdim{C}_R(M)<\infty$.
\begin{enumerate}[\quad\rm(a)]
\item \label{perfectitem1}
There is an inequality $\grade_R(M)\leq\gkdim{C}_R(M)$,
and $M$ is $\text{G}_C$-perfect of grade $g$ if and only if
$\grade_R(M)=\gkdim{C}_R(M)=g$.
\item \label{perfectitem2}
If $M$ is $\text{G}_C$-perfect of grade $g$, then so is the $R$-module
$\ext^g_R(M,C)$, and there is an isomorphism
$M\cong\ext^g_R(\ext^g_R(M,C),C)$.
\end{enumerate}
\end{cor}

\begin{proof}
Part~\eqref{perfectitem1}
is established in the next sequence; the first  equality is by definition
\begin{align*}
\grade_R(M)
&=\depth_{\ann_R(M)}(R)\\
&=\depth_{\ann_R(M)}(C) \\
&=\inf\{n\geq 0\mid\ext^n_R(M,C)\neq 0\}\\
&\leq\sup\{n\geq 0\mid\ext^n_R(M,C)\neq 0\}\\
&=\gkdim{C}_R(M).
\end{align*}
The second equality follows from the fact that a sequence in $R$
is $R$-regular if and only if it is $C$-regular; see~\cite[p.\ 68]{golod:gdagpi}.  
The third equality is standard,
the inequality is trivial, and the last equality is in~\cite[(2.1)]{frankild:rrhffd}.

Part~\eqref{perfectitem2} follows immediately from Proposition~\ref{perfect1}\eqref{perfect1item1};
see Example~\ref{tiltex1}.
\end{proof}

\section{Balanced Properties for Relative Cohomology}\label{sec2}

\begin{defn} \label{balance00}
Fix subcategories $\catx'\subseteq \proprescatx$ and $\caty'\subseteq\propcorescaty$.
We say that $\xaext$ and $\ayext$ are \emph{balanced} on
$\catx'\times\caty'$ if the following condition holds:
For each object $M$ in $\catx'$ and $N$ in $\caty'$, if $X\to M$ is a proper
$\catx$-resolution, and $N\to Y$ a proper $\caty$-coresolution, then the induced
morphisms of complexes
$$
\ahom(M,Y)\to
\ahom(X,Y)\from
\ahom(X,N)
$$
are quasiisomorphisms.  
\end{defn}

\begin{disc}
Fix objects $M\in \catx'$ and $N\in \caty'$.
If
$\xaext$ and $\ayext$ are balanced on
$\catx'\times\caty'$, then
$\xaext^n(M,N)\cong\ayext^n(M,N)$
for all  and all $n\in\zz$.
\end{disc}

The next four lemmata are tools for the proof of the Main Theorem
of this paper.

\begin{lem} \label{perp02}
Assume 
$\catw\perp\catv$.
\begin{enumerate}[\quad\rm(a)]
\item \label{perp02item1}
If $\waext^{\geq1}(\finrescatw,\catv)=0$ and $\catw\perp\catw$, then
$\finrescatw\perp\catv$.
\item \label{perp02item2}
If $\avext^{\geq1}(\catw, \fincorescatv)=0$ and $\catv\perp\catv$, then
$\catw\perp\fincorescatv$.
\end{enumerate}
\end{lem}

\begin{proof}
We prove part~\eqref{perp02item1};  part~\eqref{perp02item2}  is verified similarly.
Fix objects $M$ in $\finrescatw$ and $V$ in $\catv$ and set $n=\wpd(M)$.
We proceed by induction on $n$.
If $n=0$, then $\aext^{\geq 1}(M,V)=0$ since $\catw\perp\catv$.
So assume $n\geq 1$.
There exists an exact sequence
\begin{equation} \label{exseq03}
0\to M'\xra{\epsilon} W\to M\to 0
\end{equation}
such that $W$ is an object in $\catw$ and $\wpd(M')=n-1$.
The induction hypothesis implies $\aext^{\geq 1}(M',V)=0$.
Fix an integer $i\geq 1$.
Using the hypothesis
$\catw\perp\catv$, a standard dimension-shifting argument yields
$0=\aext^i(M',V)\cong \aext^{i+1}(M,V)$, so it remains to show
$\aext^{1}(M,V)=0$.

By Lemma~\ref{gencat01} we know $\catw\perp\catw$ implies 
$\catw\perp\finrescatw$.  Hence, the sequence~\eqref{exseq03}
is $\ahom(\catw,-)$-exact by Lemma~\ref{perp03}\eqref{perp03item2}.  
By assumption, we have
$\waext^{\geq1}(M,V)=0$
and so the long exact sequence in $\waext(-,V)$ associated to~\eqref{exseq03}
has the form
\begin{equation*}
0\to \ahom(M,V)\to \ahom(W,V)\xra{\ahom(\epsilon,V)} \ahom(M',V)\to  0.
\end{equation*}
Thus, the map $\ahom(\epsilon,V)$ is surjective.  
The assumption $\catw\perp\catv$ implies that the
long exact sequence in $\aext(-,V)$ associated to~\eqref{exseq03} starts as
\begin{equation*}
0\to \ahom(M,V)\to \ahom(W,V)\xra{\ahom(\epsilon,V)} \ahom(M',V)\to  \aext^1(M,V)\to 0.
\end{equation*}
Since $\ahom(\epsilon,V)$ is surjective, this implies $\aext^1(M,V)=0$
as desired.
\end{proof}

\begin{lem} \label{perp04}
Let $\catw$ be a cogenerator for $\catx$ and 
let $\catv$ be a generator for $\caty$.  Assume that
$\catx$ and $\caty$ are closed under extensions.
\begin{enumerate}[\quad\rm(a)]
\item \label{perp04item1}
If $\catw\perp\catw$ and $\catx\perp\catv$ and 
$\waext^{\geq 1}(\finrescatw,\catv)=0$,
then $\finrescatx\perp\catv$.
\item \label{perp04item2}
If $\catv\perp\catv$ and $\catw\perp\caty$ and
$\avext^{\geq 1}(\catw,\fincorescatv)=0$,
then $\catw\perp\fincorescaty$.
\end{enumerate}
\end{lem}

\begin{proof}
We prove part~\eqref{perp04item1}; the proof of part~\eqref{perp04item2}
is dual.
Fix an object $M\in\finrescatx$ and, using
Lemma~\ref{xhatcor}\eqref{xhatitem05'},  a $\catw\catx$-hull
$$0\to M\to K'\to X'\to 0.$$
Because $X'$ is in $\catx$, we have $X'\perp\catv$.
Lemma~\ref{perp02}\eqref{perp02item1}
implies $K'\perp\catv$ and so Lemma~\ref{perp03}\eqref{perp03item1} guarantees
$M\perp\catv$, as desired.
\end{proof}

Lemma~\ref{xhatcor} provides
the existence of the proper resolutions and coresolutions in the next two lemmata
which are the primary tools for proving the Main Theorem.

\begin{lem} \label{balance05}
Assume that $\catx$ and $\caty$ are closed under extensions, 
$\catw$ is an injective cogenerator for $\catx$, 
$\catv$ is a projective generator for $\caty$,
$\catw\perp\caty$ and $\catx\perp\catv$.
\begin{enumerate}[\quad\rm(a)]
\item \label{balance05item1}
Assume $\waext^{\geq1}(\finrescatw,\catv)=0$.
If $M$ is an object in $\finrescatx$ with proper $\catx$-resolution
$X\to M$,
then $X^+$ is $\ahom(-,\caty)$-exact.
\item \label{balance05item2}
Assume $\avext^{\geq1}(\catw, \fincorescatv)=0$.
If $N$ is an object in $\fincorescaty$ with proper $\caty$-coresolution
$N\to Y$,
then $^+Y$ is $\ahom(\catx,-)$-exact.
\end{enumerate}
\end{lem}

\begin{proof}
We proof part~\eqref{balance05item1}; the proof of~\eqref{balance05item2} is dual.
Lemma~\ref{xhatcor}\eqref{xhatitem05'} yields
a strict $\catw\catx$-resolution
$X'\to M$, and Lemma~\ref{xhat}\eqref{xhatitem02} implies that this resolution
is $\catx$-proper.  Lemma~\ref{rel01}\eqref{rel01item1} shows that $X$ and $X'$ are homotopy
equivalent, so we may replace $X$ with $X'$ to assume that $X\to M$ 
is a strict $\catw\catx$-resolution.

Fix an object $Y\in\caty$.
For each $n$, set $M_n=\coker(\partial_{n+2}^X)$,
noting  $M_{-1}\cong M$.  When $n\geq 0$, we have
 $\wpd(M_n)<\infty$ 
and we consider the exact sequences
\begin{equation} \label{exseq01}
0\to M_{n}\xra{\gamma_n} X_{n}\to M_{n-1}\to 0.
\end{equation}
It suffices to show that each of these sequences is
$\ahom(-,Y)$-exact, that is, that the map 
$\ahom(\gamma_n,Y)\colon\ahom(X_{n},Y)\to\ahom(M_{n},Y)$
is surjective.
Since $\catv$ is a generator for $\caty$ and $Y$ is 
in $\caty$, there is an exact sequence
\begin{equation} \label{exseq02}
0\to Y'\to V\xra{\tau}Y\to 0
\end{equation}
such that $Y'$ is an object in $\caty$ and $V$ is an object in $\catv$.
The assumption $\catw\perp\caty$ implies that this sequence is
$\ahom(\catw,-)$-exact by Lemma~\ref{perp03}\eqref{perp03item2}.

Fix an element $\lambda\in\ahom(M_{n},Y)$.
The proof will be complete once we find $f\in\ahom(X_n,Y)$ such that
$\lambda=f\gamma_n$.
The following diagram is our guide
\begin{equation*}
\xymatrix{
&& 0\ar[r] & M_n \ar[r]^{\gamma_n}\ar[d]_>>>>>>{\lambda}\ar@{-->}[ld]_{\sigma} 
& X_n \ar[r]\ar@{-->}[lld]_<<<<<<<{\mathrm{\delta}}\ar@{-->}[ld]^<<<<{f} 
& M_{n-1} \ar[r] & 0 \\
0\ar[r] & Y'\ar[r] & V\ar[r]^{\tau} & Y \ar[r] & 0
}
\end{equation*}
wherein the top row is~\eqref{exseq01} and the bottom 
row is~\eqref{exseq02}.

Since~\eqref{exseq02} is $\ahom(\catw,-)$-exact, it yields a 
long exact sequence in $\waext(M_n,-)$ by Lemma~\ref{notation06a}\eqref{06aitem1}. 
From Lemma~\ref{balance01}\eqref{balance01item1}
we conclude $\waext^1(M_n,Y')=0$, so this long exact sequence 
begins as follows
$$0\to \ahom(M_n,Y')\to \ahom(M_n,V)\xra{\ahom(M_n,\tau)} \ahom(M_n,Y)\to 0.$$
Hence, there exists $\sigma\in\ahom(M_n,V)$ such that $\lambda=\tau\sigma$.

Lemma~\ref{perp04}\eqref{perp04item1}
implies $\aext^1(M_{n-1},V)=0$, so an application of $\aext(-,V)$
to the sequence~\eqref{exseq01} yields the next exact sequence
$$0\to \ahom(M_{n-1},V)\to \ahom(X_n,V)\xra{\ahom(\gamma_n,V)} \ahom(M_n,V)\to 0.$$
Hence, there exists $\delta\in\ahom(X_n,V)$ such that $\sigma=\delta\gamma_n$.
It follows that
$$(\tau\delta)\gamma_n=\tau\sigma=\lambda$$
and so $f=\tau\delta\in\ahom(X_n,V)$ has the desired property.
\end{proof}

\begin{lem} \label{balance04}
Assume that $\catx$ and $\caty$ are closed under extensions, 
$\catw$ is an injective cogenerator for $\catx$, 
$\catv$ is a projective generator for $\caty$,
$\catw\perp\caty$ and $\catx\perp\catv$.
\begin{enumerate}[\quad\rm(a)]
\item \label{balance04item1}
Let $M$ be an object in $\finrescatx$ with proper $\catx$-resolution
$\alpha\colon X\to M$.  If $Y'$ is a bounded above complex of objects in $\caty$
and $\waext^{\geq1}(\finrescatw,\catv)=0$,
then the induced map $\ahom(M,Y')\to\ahom(X,Y')$ is a quasiisomorphism.
\item \label{balance04item2}
Let $N$ be an object in $\fincorescaty$ with proper $\caty$-coresolution
$\alpha\colon N\to Y'$.  If $X'$ is a bounded below complex of objects in $\catx$
and $\avext^{\geq1}(\catw, \fincorescatv)=0$,
then the induced map $\ahom(X',N)\to\ahom(X',Y)$ is a quasiisomorphism.
\end{enumerate}
\end{lem}

\begin{proof}
We proof part~\eqref{balance04item1}; the proof of~\eqref{balance04item2} is dual.
Lemma~\ref{balance05}\eqref{balance05item1}
shows that the complex $\ahom(X^+,Y_n)$ is exact for each $n$, and
a standard argument demonstrates that
$\ahom(X^+,Y)$ is exact.
From the following isomorphisms of complexes
$$\cone(\ahom(\alpha,Y))\cong\shift\ahom(\cone(\alpha),Y)
\cong\shift\ahom(X^+,Y)\simeq 0$$
one concludes that $\ahom(\alpha,Y)$ 
is a quasiisomorphism.
\end{proof}

The next result contains the Main Theorem from the introduction.

\begin{thm} \label{balance02}
Assume that $\catx$ and $\caty$ are  closed under extensions,
$\catw$ is an injective cogenerator for $\catx$, 
$\catv$ is a projective generator for $\caty$,
$\catw\perp\caty$, $\catx\perp\catv$ and
$\waext^{\geq1}(\finrescatw,\catv)=0=\avext^{\geq1}(\catw, \fincorescatv)$.
Then $\xaext$ and $\ayext$ are balanced on $\finrescatx\times\fincorescaty$.
In particular,  there are isomorphisms
$\xaext^n(M,N)\cong\ayext^n(M,N)$
for all objects $M$ in $\finrescatx$ and $N$ in $\fincorescaty$ and for all $n\in\zz$.
\end{thm}

\begin{proof}
Fix objects $M$ in $\finrescatx$ and $N\in \fincorescaty$.
Using Lemma~\ref{xhatcor}, we have a proper $\catx$-resolution
$\alpha\colon X\to M$ and a 
proper $\caty$-coresolution
$\beta\colon N\to Y$.  
Lemma~\ref{balance04} implies that the induced morphisms
$$
\ahom(M,Y)\xra{\ahom(\alpha,Y)}
\ahom(X,Y)\xla{\ahom(X,\beta)}
\ahom(X,N)
$$
are quasiisomorphisms, and hence the desired conclusion.  
\end{proof}

\begin{disc} \label{balance06}
Under the hypotheses of Theorem~\ref{balance02},
it follows almost immediately from Proposition~\ref{detect01} that
$\waext$ and $\avext$ are balanced on $\finrescatw\times\fincorescatv$.
This conclusion also follows from the weaker hypothesis
$\waext^{\geq1}(\finrescatw,\catv)=0=\avext^{\geq1}(\catw, \fincorescatv)$
using~\cite[(8.2.14)]{enochs:rha}.
\end{disc}

The next result  follows from Lemma~\ref{cogen01}
and Thoerem~\ref{balance02}.

\begin{cor} \label{balance03}
For $n=0,1,2,\ldots$, let $\catx_n$ and $\caty_n$ be subcategories of $\cata$
such that 
$\catx_n$ and $\caty_n$ are closed under extensions when $n\geq 1$.
Assume that $\catx_n$ is an injective cogenerator for $\catx_{n+1}$ and 
$\caty_n$ is a projective generator for $\caty_{n+1}$ for each $n\geq 0$.
Assume $\catx_n\perp\caty_0$ and $\catx_0\perp\caty_n$ for each $n\geq 0$.
If
$\ext^{\geq1}_{\catx_0\cata}(\operatorname{res}\comp{\catx_0},\caty_0)=0
=\ext^{\geq1}_{\cata\caty_0}(\catx_0, \operatorname{cores}\comp{\caty_0})$,
then
$\ext_{\catx_m\cata}$ and $\ext_{\cata\caty_n}$ are balanced on 
$\operatorname{res}\comp{\catx_m}\times
\operatorname{cores}\comp{\caty_n}$ for each $m,n\geq 0$. \qed
\end{cor}

We conclude with
special cases of Theorem~\ref{balance02} for our categories of interest.

\begin{notation} \label{notn1}
We simplify our notation for certain relative cohomology functors
and for some of the connecting maps from Definition/Notation~\ref{rel02a}
\begin{align*}
\pcext^n(-,-)&=\ext^n_{\catpc(R) R}(-,-)
& \icext^n(-,-)&=\ext^n_{R\,\catic(R)}(-,-)  \\
\gpcext^n(-,-)&=\ext^n_{\catgpc(R) R}(-,-)
&\gicext^n(-,-)&=\ext^n_{R\,\catgic(R)}(-,-) \\
\gpext^n(-,-)&=\ext^n_{\catgp(R) R}(-,-)
&\giext^n(-,-)&=\ext^n_{R\,\catgi(R)}(-,-) \\
\pccomp^n&=\varkappa_{\catpc(R) R}^n
& \iccomp^n&= \varkappa_{R\,\catic(R)}^n. 
\end{align*}
\end{notation}

We now show how Theorem~\ref{balance02} recovers~\cite[(3.6)]{holm:gdf}.

\begin{cor} \label{holm1}
If $R$ is a commutative ring, then
$\gpext$ and $\giext$ are balanced on
$\finrescatgpr\times\fincorescatgir$.
\end{cor}

\begin{proof}
Set
$\catx=\catgp(R)$, 
$\caty=\catgi(R)$,
$\catw=\catp(R)$ and
$\catv=\cati(R)$.
From~\cite[(2.5),(2.6)]{holm:ghd} we know that
$\catx$ and $\caty$ are  closed under extensions.
Fact~\ref{ICPG} implies that 
$\catw$ is an injective cogenerator for $\catx$ and
$\catv$ is a projective generator for $\caty$.
Clearly, we have
$\catw\perp\caty$ and $\catx\perp\catv$.
The natural isomorphisms
$$\ext^n_{\catp(R)\catm(R)}(-,-)\cong
\ext^n_{R}(-,-)\cong
\ext^n_{\catm(R)\cati(R)}(-,-)$$
from Remark~\ref{rel03} yield
$$\waext^{\geq1}(\finrescatw,\catv)=0=\avext^{\geq1}(\catw, \fincorescatv).$$
Hence, Theorem~\ref{balance02} yields the desired conclusion.
\end{proof}

The next lemmata are for use in Corollary~\ref{more04}.

\begin{lem} \label{tak01}
Let $R$ be a commutative ring and let $B$ and $B'$ be semidualizing
$R$-modules.  If $\tor^R_{\geq 1}(B,B')=0$, then
$\catpb(R)\perp\cati_{B'}(R)$.
\end{lem}

\begin{proof}
Let $P$ be a projective $R$-module and $I$ an injective $R$-module.
For each $i\geq 1$, the first isomorphism in the following sequence is 
a standard form of adjunction using the fact that $P$ is projective and $I$ is injective
\begin{align*}
\ext^i_R(P\otimes_R B,\hom_R(B',I))
&\cong\hom_R(\tor_i^R(P\otimes_R B,B'),I)\\
&\cong\hom_R(P\otimes_R \tor_i^R(B,B'),I)\\
&=0.
\end{align*}
The second isomorphism follows from the fact that $P$ is projective,
and the
vanishing is by assumption.
\end{proof}

The next example shows how to construct semidualizing $R$-modules
satisfying the hypotheses of Lemma~\ref{tak01}.

\begin{ex} \label{more01}
Let $R$ be a commutative ring and let 
$B$ and $C$ be semidualizing $R$-modules.
One has $C\in\catbb(R)$ if and only if
$B\in\catgc(R)$ by~\cite[(3.14)]{sather:crct}.
Assume  $C\in\catbb(R)$.
From~\cite[(2.11)]{christensen:scatac},
we conclude that the $R$-module $\bdc=\hom_R(B,C)$ is semidualizing, 
and~\cite[(3.1.b)]{frankild:rrhffd} yields
$\bdc\in\catab(R)$ and $B\in\catabdc(R)$.
In particular, we conclude $\tor^R_{\geq 1}(B,\bdc)=0$.

For example, 
one always has $C\in\catb_R(R)=\catm(R)$.
If $R$ is Cohen-Macaulay
and $D$ is dualizing, then $D\in\catbc(R)$.
For discussions of methods for generating 
other semidualizing modules
$B$ and $C$ such that $C\in\catbb(R)$, 
see~\cite{frankild:rrhffd,frankild:sdcms,sather:divisor}.
\end{ex}

\begin{lem} \label{more02}
Let $R$ be a commutative ring and let $B$ and $C$ be semidualizing
$R$-modules such that $C\in\catbb(R)$. With $\bdc=\hom_R(B,C)$,
there are containments
$\finrescatpbr\subseteq\catbb(R)\cap\catabdc(R)\supseteq
\fincorescatibdcr$.
\end{lem}

\begin{proof}
We verify the first containment; the second one is dual.
Fact~\ref{projac} implies $\finrescatpbr\subseteq\catbb(R)$. 
From Example~\ref{more01}, we have $B\in\catabdc(R)$,
and this readily implies
$\catpb(R)\subseteq\catabdc(R)$.
Fact~\ref{projac} then yields 
$\finrescatpbr\subseteq\catabdc(R)$.
\end{proof}

\begin{lem} \label{more03}
Let $R$ be a commutative ring and let $B$ and $C$ be semidualizing
$R$-modules such that $C\in\catbb(R)$. If $\bdc=\hom_R(B,C)$,
then $\pbext$ and $\ibdcext$
are balanced on $\finrescatpbr\times\fincorescatibdcr$.
\end{lem}

\begin{proof}
Let $M$ and $N$ be $R$-modules with
$\pbpd_R(M)<\infty$ and $\ibdcid_R(N)<\infty$.
From Lemma~\ref{more02} we conclude $M,N\in\catbb(R)\cap\catabdc(R)$
and so~\cite[(4.1)]{sather:crct} implies that the following
natural maps are isomorphisms for each $n\in\zz$
$$\pbext^n(M,N)\xra[\cong]{\pccomp^n(M,N)}\ext^n_R(M,N)
\xla[\cong]{\ibdccomp^n(M,N)}\ibdcext^n(M,N).$$
In particular, we have
$$\pbext^n(\finrescatpbr,\catibdc(R))=0=
\ibdcext^n(\catpb(R),\fincorescatibdcr)$$
and the desired conclusion follows from~\cite[(8.2.14)]{enochs:rha}.
\end{proof}

Theorem~\ref{balance02} and Lemma~\ref{more03} yield
the next result. 

\begin{cor} \label{more04}
Let $R$ be a commutative ring and let $B$ and $C$ be semidualizing
$R$-modules such that $C\in\catbb(R)$. Set $\bdc=\hom_R(B,C)$
and assume $\catpb(R)\perp\catgibdc(R)$ and $\catgpb(R)\perp\catibdc(R)$.
Then 
$\gpbext$ and $\gibdcext$
are balanced on $\finrescatgpbr\times\fincorescatgibdcr$.
\qed
\end{cor}

\begin{question} \label{more04'}
Let $R$ be a commutative ring and let $B$ and $C$ be semidualizing
$R$-modules such that $C\in\catbb(R)$. With $\bdc=\hom_R(B,C)$,
must one have
$\catpb(R)\perp\catgibdc(R)$ and $\catgpb(R)\perp\catibdc(R)$?
\end{question}
 
If the answer to this question is ``yes'' then the assumptions
$\catpb(R)\perp\catgibdc(R)$ and $\catgpb(R)\perp\catibdc(R)$
can be removed from Corollary~\ref{more04}.
Next we discuss one case where this is known,
showing that~\cite[(5.7)]{sather:crct} is a special case of Corollary~\ref{more04}.

\begin{disc} \label{more07}
Let $R$ be a commutative Cohen-Macaulay ring with a dualizing module $D$.
Let $B$  be a  semidualizing
$R$-module.  The membership $D\in\catbb(R)$ is in~\cite[(4.4)]{christensen:scatac}.
The conditions
$\catpb(R)\perp\catgi_{B^{\dagger_D}}(R)$ and $\catgpb(R)\perp\cati_{B^{\dagger_D}}(R)$
follow from the containments
$\catgi_{B^{\dagger_D}}(R)\subseteq\catbb(R)$
and 
$\catgpb(R)\subseteq\cata_{B^{\dagger_D}}(R)$
in~\cite[(4.6)]{holm:smarghd}.
It follows that  $\gpcext$ and $\ext_{\catgi_{C^{\dagger_D}}}$ are balanced on 
$\finrescatgpc\times\operatorname{cores}\comp{\catgi_{C^{\dagger_D}}(R)}$.
\end{disc}

The following question is from
the folklore of this subject and is related to the composition question for ring homomorphisms
of finite G-dimension; see~\cite[(4.8)]{avramov:rhafgd}.
Remark~\ref{more05} addresses 
its relevance  to Corollary~\ref{more04} and Question~\ref{more04'}.

\begin{question} \label{more06}
Let $R$ be a commutative ring and let $B$ and $C$ be semidualizing
$R$-modules such that $C\in\catbb(R)$.  Must the following containments 
hold?
\begin{align*}
\catgpb(R)&\subseteq\catgpc(R)
&\catgib(R)&\subseteq\catgic(R) \\
\catac(R)&\subseteq\catab(R)
&\catbc(R)&\subseteq\catbb(R)
\end{align*}
\end{question}

\begin{disc} \label{more05}
Let $R$ be a commutative Cohen-Macaulay ring with a dualizing module $D$.
Let $B$ and $C$ be semidualizing
$R$-modules such that $C\in\catbb(R)$.
Arguing as in~\cite[(3.9)]{frankild:rrhffd}, one concludes
$B^{\dagger_D}\in\catbbdc(R)$ and
$B\in\catb_{B^{\dagger_C\dagger_D}}(R)$.
Assume that the answer to Question~\ref{more06} is ``yes''.
Then there are containments
\begin{align*}
\catgpb(R)&\subseteq\cata_{B^{\dagger_D}}(R)
\subseteq\cata_{B^{\dagger_C}}(R)
&\catgibdc(R)&\subseteq\catb_{B^{\dagger_C\dagger_D}}(R)
\subseteq\catbb(R)
\end{align*}
by~\cite[(4.6)]{holm:smarghd}.
One concludes 
$\catpb(R)\perp\catgibdc(R)$ and $\catgpb(R)\perp\catibdc(R)$
from the easily verified conditions
$\catpb(R)\perp\catbb(R)$ and $\catabdc(R)\perp\catibdc(R)$.

In particular, if the answer to Question~\ref{more06} is ``yes'',
then the same is true for Question~\ref{more04'} and 
the assumptions
$\catpb(R)\perp\catgibdc(R)$ and $\catgpb(R)\perp\catibdc(R)$
can be removed from Corollary~\ref{more04}.
\end{disc}



\providecommand{\bysame}{\leavevmode\hbox to3em{\hrulefill}\thinspace}
\providecommand{\MR}{\relax\ifhmode\unskip\space\fi MR }
\providecommand{\MRhref}[2]{%
  \href{http://www.ams.org/mathscinet-getitem?mr=#1}{#2}
}
\providecommand{\href}[2]{#2}

\end{document}